\documentclass[11pt]{amsart}
\usepackage{graphicx,color}
\usepackage[utf8]{inputenc}
\newtheorem{theorem}{Theorem}[section]
\newtheorem{lemma}[theorem]{Lemma}

\newtheorem{corollary}{Corollary}[theorem]
\newtheorem*{proposition*}{proposition}
\newtheorem*{theorem*}{Theorem}
\newtheorem*{lemma*}{Lemma}
\theoremstyle{definition}
\newtheorem{definition}[theorem]{Definition}
\newtheorem{example}[theorem]{Example}

\usepackage[utf8]{inputenc}
\theoremstyle{remark}
\newtheorem{remark}[theorem]{Remark}
\usepackage[english]{babel}
\usepackage[T1]{fontenc}
\usepackage{appendix}
\usepackage{ifthen}
\usepackage{mathrsfs}
\usepackage{psfrag}
\numberwithin{equation}{section}

\begin{document}
\title{SAT Solving by Multiplication in and Addition out of Clauses}
\author{Pelle Brooke Borgeke}
\address {Linn\oe us university}
\curraddr{}
\email{pelle.borgeke@lnu.se}
\subjclass[2010]{primary}
\date {6 November 2026}
\begin{abstract} In this article, we use the PBNF-transform, or Polynomial Boolean Normal Form, described earlier in [2] and [3] and now extended to higher arity, to solve Boolean satisfiability problems (SAT), which are part of the P versus NP problem. SAT consists of a statement of conjunctions of clauses $S^j_{c(\lor,\ell\geq 3)}$, and in our SAT-solving models, we dualize the logical equations, inverting the problem by \emph{multiplying} letters (literals) in the transferred clauses by $$g(S^j_{c(\land(\lor),\ell\geq 3)})=s_g^j= pqrs\cdots s_{\ell},$$ and using a binary iterated composition based on \emph{adding polynomials} for the conjunctions in our solution formula. We obtain Polynomial Boolean Equations, PBE=0, by also renaming the truth value $1\mapsto 0$ (T). With reversed literal repetition, we get the CNF as the sum of the transformed clauses. In fact, even with $s_g^is_g^j \not = 0$, we can reduce to \emph{multiplication in} and \emph{addition out} of clauses, as we do not look for the entire kernel of PBE for a YES instance for SAT.   \end{abstract}
\maketitle
\section{Introduction: SAT-solving models and Results}
In this work, we use the PBNF\emph{-transform} (Polynomial Boolean Normal Form), described in [2] and [3] and now extended to higher arity (see Section 2), in an application to construct solving models for the Boolean satisfiability problem, called SAT, studied in computer science and logic. It is a logical form written in CNF, "ANDs of ORs", a statement of conjunctions binding a fixed but unbounded number of clauses $S^j_{c(\land(\lor),\ell\geq 3)}$ in the logical operator space $\mathcal{OP}$, both in length and number of literals ($\ell$); classical, $c$, with truth values 1 and 0 as follows, and here written in Left Side (1.1) and Right side (1.2), to show the planed composition in (1.8) \begin{align} S^1_{c(\lor,\ell \geq 3)} \land S^2_{c(\lor,\ell\geq 3)} \land \ldots \land S^{j-1}_{c(\lor,\ell\geq 3)} \land S^j_{c(\lor,\ell\geq 3)}\land\ldots \land S^{n-1}_{c(\lor,\ell\geq 3)} \land S^n_{c(\lor,\ell\geq 3} \\  \land \  S^{n}_{c(\lor,\ell\geq 3)} \land  S^{n-1}_{c(\lor,\ell\geq 3)}\land \ldots \land  S^j_{c(\lor,\ell\geq 3)} \land  S^{j-1}_{c(\lor,\ell\geq 3} \land \ldots  S^2_{c(\lor,\ell\geq 3)}\land  S^1_{c(\lor,\ell\geq 3)}.\end{align} SAT is an example of the \emph{P versus NP problem} [5], usually informally stated as: "Whether every problem whose solution can be quickly verified can also be quickly solved." Algorithms to solve this problem use computer programs [7], such as DPLL, CDCL, so-called SAT solvers. CNF is also converted, e.g., in [6] and [10], into Boolean polynomial representations as ANF, the methods directly convert disjunctions in the clauses using the Boolean join, "AND of ORs". This formulation gives problems described in [10] as a \emph{tiling mismatch}: the "pieces" from CNF, as described in [6] as "Block-wise conversion from CNF to ANF", do not exactly match the "pieces" from ANF ("XOR of ANDs").
Here we suggest a PBNF-based \emph{inverting} description, an "OR of ANDs", or, for SAT-solving, a "XOR of ANDs", a \emph{multiplication in} and \emph{addition out} of the clauses instead.  We dualize truth values, $1$ (T) $\mapsto 0$, and the logical forms $\lor \mapsto \land$, $\land \mapsto \lor$, inverting the inputs, but we keep $\lnot$ fixed. This means turning clauses into phrases \begin{align}g: \text{CNF} \ \to g(p',q',0) \in \mathcal{PB} \ \text{and}\ g(\hat S_{c(\lor,\ell \geq 3)})= \hat s_g=\sum_{\ell}(pqrs...s_{\ell}) ,\end{align} where $(\hat \cdot):=$ collection of objects, $g(p',q',0) \in \mathcal{PB}$ is a member in House of Polynomial Boolean Families, and $g$ the transfer function.
 After that we can use a condition of reversed repetition of literals ($\hat p\to \hat p.$ or $\hat p.\to \hat p$ to get just addition of phrases in the Boolean join $s_{g\text{(xor)}}=s_g^{i}+s_g^{j}$ as $s_{g\text{(and)}} =s_g^{i}s_g^{j}=0$\begin{align}\hat s_{g\text{(xor)}}:=s_g^1+ s_g^2+ \cdots + s_g^{i}+ s_g^{j}+ \cdots + s_g^{n-1}+ s_g^{n}=0 \end{align}  which is a \emph{Polynomial Boolean Equation}, PBE=0, representing the CNF-form, that we study. With a \emph{solution} to PBE=0 we mean a member in the kernel of PBE or ker $\hat s_{g\text{(or)}}:=\{\hat p | \ \hat s_g(\hat p)= 0 \}$. We regard the kernel as a union \begin{align} \text{ker} \ \hat s_{g\text{(or)}}\supset\text{ker} \ \hat s_{g\text{(xor)}} \cup \text{ker} \ \hat s_{g\text{(and)}},\end{align} of an \emph{xor}-kernel and an \emph{and}- kernel. If we are not looking to exploit $\text{ker} \ \hat s_{g\text{(or)}}$, as in the case of SAT, we may use (1.4) although $s_g^i s_g^j\not=0$ for solutions not beginning with a true value, as $(0,x_2,x_3\cdots x_{\ell})\in \text{ker} \ \hat s_{g\text{(and)}}$ for $g(p',q',0)$.
In Example 3.11 found in [5] we find for $\hat s_{g(\text{xor})} = pqr+p.qr.+pq.s +p.r.s.$ that   $$ \text{ker} \ \hat s_{g(\text{xor})}=\{1100, 1101, 1000, 1010\}\subset \text{ker} \ \hat s_{g(\text{or})}=\{ 0001, 0010, 0011, 0110,$$ $$ 0111, 1100, 1101, 100, 1010 \}\land \ \hat s_{g(\text{or})}= pqr+p.qr.+pq.s+ p.r.q.s. =0,$$ which is a part of $\text{ker} \ \hat s_{g(\text{or})}.$
To get a formal solution "OR of ANDs" instead of "XOR of ANDs" above we use a iterated binary composition  \begin{theorem*} (Theorem 3.14 in Section 3.) Let the Boolean satisfiable problem for a conjunct of $n\in \mathbb{N}$ clauses, called the \emph{SAT}-problem, be transferred by inversion of the $\mathcal{PB}$-family $g(p',q',0)\cong g(p,q,1)$ (see table) into multiplication for the clauses $g( S_{c(\lor,\ell \geq 3)}^j)= s_g^j=(pqrs...s_\ell),$ and for the conjunctions, a reflected polynomial composition defined by $\circ :(s^{i}_g. s^j_g.)'=s^{i}_g s^j_g +s^{i}_g+ s^j_g$ based on the iterated \emph{multiplicative} binary transfer function  $g:\mathcal{PB} \times \mathcal{PB} \to \mathcal{PB},$ acting first on simple arguments, then on composed and so on, until we compose the complete left and right sides \emph{(1.8)} to get the polynomial equation, PBE=0 in \emph{(1.9)}.\begin{align}s^{j-1}_g \circ s^{j}_g=(s^{j-1}_g+1) (s^{j}_g+1)+1=s^{j-1}_g.s^{j}_g.', \\(s^{j-1}_g. s^j_g.)' \ \text{and}\ (s^{j+1}_g. s^{j+2}_g.)'= ((s^{j-1}_g. s^j_g.).). ((s^{j+1}_g. s^{j+2}_g.).)'\end{align}  \begin{align}(\cdot \cdot(((((s_g^1. s_g^2.). \cdot \cdot   (s_g^{n-1}. s_g^n.).).). \circ (((((s_g^n.  s_g^{n-1}.). \cdot \cdot (s_g^2. s_g^1.).).\cdot \cdot )' \\ ((\cdot \cdot(((((s_g^1. s_g^2.). \cdot \cdot   (s_g^{n-1}. s_g^n.).). \cdot \cdot (((((s_g^n.  s_g^{n-1}.). \cdot \cdot (s_g^2. s_g^1.).).).\cdot \cdot ).)'=0 \end{align} \end{theorem*} In Theorem 3.15 and in two Corollaries, 3.15.1 and 3.15.2, we prove other possibilities to simplify the calculations if $s_{g(\text{and})} \not = 0$. For example, the multiplicative higher-order terms are often \emph{saturated}, or contract with other terms to become saturated, meaning they are of the form $p.q.r.s.\cdots s _ {\ell}. $, so every literal is negated, giving a logical argument in a final step. Other possibilities are different kinds of conjugations, for example, the form $s_g^{i}s_g^{j}. +s_g^{j}$ could be conjugated from the right to get $s_g^{j}=0$ to solve first, and then use the variables in $s_g^{i}=0$ to complete a solution.
\begin{scriptsize} \begin{center}
\begin{tabular}  {||c | c |c |c |c ||}
\hline
Op$_{1} \mapsto $  & PB $ f(p|q, 1)$&PB$' f(p|q,1')$ & PB$_{(.)}
f(p'|q',1)) $ & PB$_{(.)}' f(p'|q',1')$ \\ [0.5ex]
\hline\hline
$ = p \mapsto p+0$ & $p(10)$ &$p(01)$ & $p(01)$& $p(10)$\\
\hline
$ \lnot p\mapsto p+1 $ & $p(01)$ &$p(10)$ & $p(10)$& $p(01)$\\
\hline
$ - p \mapsto p+p$ & $p(00)$ &$p(11)$ & $p(11)$& $p(00)$\\
\hline
$ +p\mapsto p+p'  $ & $p(11)$ &$ p(00)$ & $p(00)$& $p(11)$\\
\hline
\hline
Op$_{2, p+q} \mapsto$  & PB $ f(p+q, 1)$&PB$' f(p+q,0)$ & PB $_{(..)}
f(p'+q',1)) $ & PB $_{(..)}' f(p'+q',0)$ \\ [0.5ex]
\hline\hline
$ p $ & $p(1100)$ &$p'$ & $p'$& $p$\\
\hline
$ p' $ & $p'(0011)$ &$p$ & $p$& $p'$\\
\hline
$ q $ & $q(1010)$ &$q'$ & $q'$& $q$\\
\hline
$ q' $ & $q'(0,1,0,1)$ &$ q$ & $q$& $q'$\\
\hline
$\hat p| \hat q$ & $\hat p+\hat q $ & $(\hat p+\hat q)'$ & $\hat p+ \hat q$ & $(\hat p+\hat q)'$ \\
\hline
$\hat p\Leftrightarrow \hat q$ & $(\hat p+ \hat q)' $ & $\hat p+\hat q $ & $(\hat p+\hat q)'$ & $\hat p+\hat q $ \\
\hline
\hline
Op$_{2, \hat p \hat q} \mapsto$  & PB $ g(\hat p,\hat q,1)$& PB$' g(\hat p, \hat q,0)$ & PB$_{(..)}g(\hat p', \hat q', 1)) $ & PB$_{(..)}'  g(\hat p',\hat q',0)$ \\ [0.5ex]
\hline
\hline
$\hat p\lor \hat q $& $\hat p.\hat q.'$ &$\hat p.\hat q.$ & $\hat p\hat q'$& $\hat p\hat q$ \\
\hline
$\hat p\land \hat q$ & $\hat p\hat q$ & $\hat p\hat q'$ & $\hat p.\hat q.$& $\hat p.\hat q.'$\\
\hline
\hline
$\hat p\Rightarrow \hat q$ & $\hat p \hat q.'$ & $\hat p\hat q.$ & $\hat p.\hat q'$ & $\hat p.\hat q$ \\
\hline
$\hat p\not\Rightarrow \hat q$ & $\hat p \hat q.$ & $\hat p\hat q.'$ & $\hat p.\hat q$ & $\hat p.\hat q'$ \\
\hline
$\lnot \hat p\Rightarrow \lnot \hat q$ & $\hat p.\hat q'$ & $\hat p.\hat q$ & $ \hat p \hat q.'$ & $ \hat p \hat q.$ \\
\hline
$\hat p\Leftarrow \hat q$ & $\hat p.\hat q'$ & $\hat p.\hat q$ & $\hat p \hat q.'$ & $\hat p\hat q.$ \\
\hline
$\hat p\not\Leftarrow \hat q$ & $\hat p.\hat q$ & $\hat p.\hat q'$ & $\hat p \hat q.'$ & $\hat p\hat q.$ \\
\hline
$ \lnot \hat p\Leftarrow \lnot \hat q$ & $\hat p\hat q.'$ & $\hat p\hat q.$ & $\hat p.q'$ & $\hat p.q$ \\
\hline
\hline
$\hat p\downarrow \hat q$ & $\hat p. \hat q.$ & $\hat p.\hat q.'$ & $\hat p\hat q$& $\hat p\hat q'$\\
\hline
$\hat p\uparrow q$ & $\hat p \hat q'$ & $\hat p \hat q$ & $\hat p.\hat q.'$ & $\hat p.\hat q.$ \\ [1ex]
\hline
\end{tabular} \end{center} \end{scriptsize}
\section {The PBNF-transform and the Binary Extension }
In the table above, we show the polynomial symbols for the logical operators, now, in a new shorthand compact writing, both for the singular linear operators $p \to p =,\lnot,-,\oplus, \ a \text{Klein four-group, see [2]}, $ on the top, then binary linear operators in an Abelian 6-group, defined by $+:p \times q \to p+q; \ ,$ and finally a binary quadratic Commutative Boolean coordinate ring, defined by $$+, \cdot :p \times q \to p+q, p\cdot q. $$ The notation is standard where there is a standard for the logical operators, which is not the case for the singular operators, for the binary here $p|q =: p$(xor)$q$, for the singular we use  $\oplus$ as we use the additive operation +. For the complement $p'=(p+1)$ and $q'=(q+1)$, while we, in the example $\hat p\lor \hat q \mapsto (\hat p+1)(\hat q+1)+1=\hat p.\hat q.' $ use an inner $". "$ (dot) complementation, and $"'"$ ends the formula, and the "hat" is defined below. Op denotes different operator spaces and observe that we also can use dots in an Op space instead of (inner) parentheses, as in the examples to come.
To address higher-arity problems, we extend our propositional calculus formulation, a bijective transfer of the ordinary 16 binary logic operators, which we describe in [2] and [3] and write above in a more suitable way for this purpose. In the table, we fix $\ell \in \mathbb{N}$ as the number of letters (unbounded arity) and extend to $\hat p$ and to $\hat q$, now multinomials, to mean any collection of monomials. We define the \emph{polynomial bundles} $[g_*] $ as equivalence classes that implement the logical operator $* \in$ Op in all possible modes.  \begin{align} [g_{*}] =  g(\hat p * \hat q): \sum_{p,q,r,s \cdots s_{\ell}}= \hat p\hat q + \hat p+\hat q+ a ,  \text{under \ the \ choise of} \\ \nonumber \text{Selectors:}  \ p_1(1100),\ p_2(1001), \ p_3(0110), \ p_4(0011), \ q_1(1010), \ q_2(0101) \\ \nonumber  \text{renaming \ of \ truth \ value \ 1|0, \ inner} \ (\cdot). \ \text{and} \ (\cdot)' \ \text{outer complementation.} \end{align} Together with the Selectors, we show above the connected binary number string.
The extension to the \emph{binary model case} is due to the Disjunctive Boolean Normal Form, DBNF(or dnf).\begin{definition} Let DBNF, Boolean Disjunctive Normal Form, be a statement form $\mathcal{S}_{c(\lor (\land ), \ell)}$ given by a disjunction of conjunctions in a standardized way, with disjunctions binding modes of collected conjunction, only using $\land, \lor$ and $\lnot$, where $\lnot$ must accompany a statement letter, as in the right side of conversions 4 and 5. below. We use the following rules, with translation to PBNF $g(p,q,1)$ in the compact way \begin{equation}\left\{
\begin{array}{ll}
1.\ \lnot(\lnot x) = x \mapsto p.'=p  \\
2.\ x \land x=x \mapsto pp=p\\
 x \lor x=x \mapsto p.p.'=p\\
3. \ x \land y=y \land x  \mapsto pq=qp \\
 \ x\lor y= \ y\lor x \mapsto  p.q.'=q.p.' \\
4. \ \lnot(x \land y)=\lnot x \lor \lnot y \mapsto pq'=(p.).(q.).' \\
5. \ \lnot(x \lor y)=\lnot x \land \lnot y  \mapsto (p.q.')'=p.q.\\
6.\ x\land(y \lor z)\equiv (x \land y) \lor (x \land z) \mapsto p(q.r.)'=(pq).(pr).'\\
7. \ x\lor(y \land z)\equiv (x \lor y) \land (x\lor z)\mapsto p.(qr).'=(p.q.).(p.r.)'\\
8. \ n x^{\ell} + (n-1)x^{\ell}=x \ \mapsto np^{\ell}+(n-1)p^{\ell}=p \ (\text{DBNF}).
\end{array}\right. \end{equation} \end{definition} Numbers 6 and 7 can be used to test the polynomial calculations with these shorthand notations. The connectives $\land, \lor$ and $\lnot $ are fundamental in Propositional Calculus (PC); in fact, they are "all we ever need", to quote Quine [9, page 43], where he discusses higher arity in PC. He concludes in the chapter named "Reduction to Three Primitives" that \emph{every mode of composition describable by a truth table is translatable into denial, conjunction, and alternation}, and we demonstrate this fact in Theorem 2.10. It is  clearly also visible in the Boolean sex-tuple with equivalence classes written below, where $\mathcal{PB}$ is the House of Polynomial Families  \begin{align}\langle \mathcal{PB}, \ \lor \cong  [\hat p.\hat q.' \hat p. \hat q. \ \hat p \hat q' \ \hat p \hat q], \land \cong [\hat p \hat q \ \hat p \hat q' \ \hat p. \hat q. \ \hat p. \hat q.'],  \ (\cdot)'  \cong \ [\hat p.\ \hat p], [01], [10] \rangle. \end{align} We note that the equivalence classes are reflections of each other, and we shall use $[\lor \ \hat p \hat q]$ and $[\land \ \hat p. \hat q.']$ for the SAT-problems (more of this later). Here we map the logical operators to the ring of polynomial functions we call a four-place Boolean coordinate ring of I$^2$ (the unit square), visualizing the coordinates as ordered areas in the unit square. Namely, DBNF is a standard representation of propositional statements of any arity to form disjunctions of conjunctions; thus, we have a mapping from the class $\mathcal{OP}$ of logical operators \begin{align}h: \mathcal{OP} \longrightarrow \ \text{DBNF, with} \ h(\hat p *\hat q)=\hat p *_{dnf}\hat q \\ \kappa: \ \text{DBNF} \longrightarrow \ \text{PBNF, with}\ \kappa(\hat p *_{dnf} \hat q)=[(\hat p+1)(\hat q+1)] \in  \mathcal{PB}\\  g=h \circ \kappa: \mathcal{OP} \longrightarrow \text{PBNF},\ \text{defined by} \ h(\hat p *\hat q )\circ \kappa(\hat p *_{dnf} \hat q)\\ = [(\hat p+1)(\hat q+1)]\in  \mathcal{PB}, \end{align}to the equivalence class $[(\hat p+1)(\hat q+1)] \in  \mathcal{PB}$. We shall use the composite function $g$ as a direct conversion $$g: \mathcal{OP}\supset \ \text{CNF} \ \to g(p',q',0) \in \ \text{PBNF}$$ but also in the binary automorphism $g:\mathcal{PB} \times \mathcal{PB} \to \mathcal{PB}$ used in an iterated polynomial composition $s^{i}_g s^j_g +s^{i}_g+ s^j_g$ that follows the initial transfer $ \text{CNF} \longrightarrow \ \text{PBNF},$ where disjunction is mapped on the multiplicative form $s^{i}_g s^j_g.$
For the background of our polynomial formulation of PC, we connect to the \emph{principles of duality}; see Church in [4], a dualization between two different operator spaces, where a dual statement can be used as an alternative in theorem proving, with the usual logical axioms and rules of inference. It looks like the following \begin{align}\text{Op} \ni (\land, \lor, \Rightarrow, \equiv, \lnot, 1) \mapsto (\lor,\land \not \Leftarrow, \not \equiv \lnot, 0) \in \text{Op}{(..)}',\end{align} for some example connectives in the normal operator space Op (T=1, F=0) transferred to the operator space $\text{Op}{(..)}'$(T=0, F=1).
We get, for example, for dualization of the conjunction above, where $\land$ is changed to $\lor$ \begin{align} \ (x\land y) \mapsto (x\land y) \lor (x \land y') \lor (x' \land y).\end{align} Next, we convert the DBNFs into linear and quadratic polynomial families in $\mathcal{PB}$ (see [2] and [3] for more information). As we have two non-homeomorphic domains for singular (I =: unit line) and binary operators (I$^2$ = : unit square) (Definition 2.3), we use a composite transfer function $f^{\text{I}}(g_{(\text{I}^2)}),$ where $f^{\text{I}} \mapsto(1|0)$ is the \emph{characteristic} function that adds 1 to the $g$-resulting polynomial, and here $\hookrightarrow$ is the \emph{fiber}(bijective return) to Op, if we do not get a numeric result,  1|0  \begin{align} f^1(g_{\text{I}^2}):\mathcal{OP} \to (\mathbb{Z}_2 |\mathcal{PB}_{(\mathbb{Z}_2)}) \ \text{defined by} \\ f^1(g_{\text{I}^2})
 =(a_1pq+a_2p+a_3q+a)(+1) \\ = \ ((1|0)
\big{|}g(p_j,q_k,a), g(p_j,q_k,a)\hookrightarrow
\mathcal{OP}, \\ a_{n=1,2,3} =0|1, a=1|0, \ j=1,2,3,4\land k=1,2. \end{align}
We then get for the normal, (T=1), polynomial family  $g_{(p_1, q_1, 1)}$ in (2.8) \begin{align}g_{(p_1,q_1,1)}(x \land y)= pq \ \text{for \ LS \ and \ for \ RS:} \\ g_{(p_1,q_1,1)}((x\land y) \lor (x\land y') \lor (x'\land y))\\ =pq+p(q+1)+(p+1)q  = pq+p+q  \hookrightarrow_{\text{Op}{(..)}'} (p \land q). \end{align} We can use addition for disjunctions in (2.15) as the conjuncts are disjoint, using the number strings (Definition 2.3) $1000+0100+0010=1110$. Now, the first polynomial (LS) is true for $p=q=1$ with T=1, and the last (RS) is true for $(11,10,01) $; this can be used as the defining property for disjunction, max $(p_j,q_k):=\lor$ and min$(p_j,q_k)=:\land$, for more on this max/min formulation, see [1]. To get LS=RS, we have to negate(reflect) both $p$ and $q$ and also the whole polynomial $pq+p+q$ in (2.16) $$(p.q.+p.+q.)'=((p+1)(q+1)+(p+1)+(q+1))'=(pq+1)'=pq$$ and that is what happened in (2.8). This could be generalized by writing $(1000 \parallel 0001 +1111 =1110)$ or \begin{align} x_1x_2x_3x_4 \parallel x_4x_3x_2x_1 +1111 =(x_4x_3x_2x_1)'.\end{align} Here we use $\parallel$ to denote reflection, and we write the output of an operator as a string of digits, $x_1x_2x_3x_4\cdots x_n$.
We conclude that we have an \emph{isomorphic} mapping $\land \mapsto \lor$ above by renaming truth values $1 \mapsto 0$, and using reflection (for the $p$-input $1100\parallel 0011$ and for the $q$-input $1010 \parallel 0101$. We use this isomorphic mapping in the SAT problems, with the family $g(p',q',0)$, to dualize clauses to phrases and conjunctions to disjunctions, using the PBNF-transform to form polynomial equations to substitute for the logical forms.
The singular operator above in (2.8) will also play an important role; we have ($\lnot) \mapsto (\lnot)$ as  $\text{Op}_1 \ni x' \mapsto ((x'){.})' \in \text{Op}{.}', $ and for the polynomial \begin{align} \lnot \mapsto  \ p+1 \mapsto ((p+1)+1)+1=p+1, \hookrightarrow \lnot \end{align} so here are two complementation, these canceling out makes $\lnot$  \emph{self-dual}, so that $\lnot$ is a fix point into $g(p',q',0)$. This means the negations in the transferred clauses remain the same, so $(\hat p \lor \lnot \hat q) \mapsto_{g(p',q',0)} \hat p \hat q.$ (see the table above, or [3]).
Adding 1 swaps the truth values: $1+1$ becomes 0; thus, the normal space becomes the complement space. This also helps us, since the polynomial equations are now zero when true, so they are in normal form after simplification and reduction.
Below, we present the connections for the four binary families in the table, now in lattices with number string, but not in compact form, to give another view $$\boldsymbol{\Downarrow} \ \text{PB}': g(p,q,0)$$ \begin{scriptsize} $$Trivial \ operator \bigcirc (1111) pq+p(q+1)+(p+1)q+(p+1)(q+1)=1$$ $$ \text{3-}area \ operators  \bigcirc  (1110) p.q.' \ A_4 ' \ \bigcirc  (1101) (p+1)q+1 \ A_3 ' \ \bigcirc  (1011) p(q+1)+1 \ A_2' \ \bigcirc  (0111) pq+1 \ \ A_1 '$$ $$ Selectors { \bigcirc   \ (1100) p \bigcirc  \ (1010) q \bigcirc (0110) p+q \ \bigcirc \ (1001) p+q+1  \bigcirc  (0101)q+1  \bigcirc \ (0011) p+1 } $$ $$ Unit\text{-}areas \bigcirc  (1000) pq \ A_1 \ \bigcirc  (0100) p(q+1) \ A_2  \ \bigcirc  (0010) (p+1)q \ A_3  \ \bigcirc  (0001) (p+1)(q+1) \ A_4  $$
$$ Trivial \ operator \bigcirc \  (0000)\ pq+p(q+1)+(p+1)q+(p+1)(q+1)+1=0$$ \end{scriptsize} $$\boldsymbol{\Uparrow} \ \text{PB} : g(p,q,1)$$ The arrows define the read direction. Here the notation $A_{n=1,2,3,4}$ of different areas refers to the geometrization of the unit square, discussed later in this Section. In the lattice, we replace the operator symbols for the connectives with polynomial \emph{symbols}, and we call the first line, $pq, p(q+1), (p+1)q$ and $(p+1)(q+1)$, the unit symbols. These 16 binary symbols capture all the properties of the connectives they represent via polynomial denotation. For example take $p(q+1)+1=pq.'$, at once we see that this is correct whenever $p=0$ (false) and when $p=1$ and $(q=1)$ for $ T=1$. We also note that $(p+1)(q+1)+1$ is true in all cases but $p=q=0$. \footnote{This can be compared to the operator symbol $\Rightarrow$, which does not contain such information, but could be regarded as a \emph{skeumorphism}, none of the other operator symbols contain any (correct) information at all; many are not even skeumorphisms, such as $\land, \lor, \uparrow, \downarrow$, etc.} The Trivial operators are given here as the sum of the unit symbols or the complement of this sum. For the whole lattice, we note that $x\leq y$ iff $xy=x$, which follows from the construction. Following $1000 - pq \to 1100 -p,$ gives $1000 \cdot 1100=1000,$ and for the polynomial $pq\cdot p=pq$. The next step gives $1100 \cdot 1110=1100,$ and $p\cdot (p+1)(q+1)+1=p.$
Now we reflect  PB$_{(..)} \parallel \text{PB}'_{(..)}.$
$$\text{PB}_{(..)}': g(p',q',0)\Downarrow $$\begin{scriptsize}$$ \bigcirc (1111) pq+p(q+1)+(p+1)q+(p+1)(q+1)=1$$
$$ \bigcirc (0111) pq+1  \quad \bigcirc (1011) p(q+1)+1 \quad \bigcirc (1101) (p+1)q+1 \quad \bigcirc (1110) (p+1)(q+1)+1\quad$$ $$ \bigcirc (0011) p+1 \quad \bigcirc (0101) q+1 \quad   \bigcirc (1001) p+q+1  \quad \bigcirc (0110) p+q  \quad \bigcirc (1010) q \bigcirc (1100) p  $$ $$ \bigcirc (0001) (p+1)(q+1)  \quad \bigcirc (0010) (p+1)q \quad \bigcirc (0100) p(q+1) \bigcirc \quad (1000) pq \quad $$ $$  \bigcirc (0000) pq+p(q+1)+(p+1)q+(p+1)(q+1)+1=0$$ \end{scriptsize}
$$\text{PB}_{(..)}: g(p',q',1)\Uparrow $$ We have changed the selectors to $p_4 (0011)$ and $q_2(0101)$, which means adding 1 to both $p$ and $q$, and for PB$_{(..)}'$ there is also a change in the reading $1'=0$.  We note that reflections $(\parallel)$ come from the complement of the inputs.
Now we discuss the general case, which is of the most theoretical interest. \begin{definition}We visualize each selector as an area $A^{\ell}$, where $\ell$ counts the number of selectors in the unit square \begin{align} \text{I}^2_{\ell}=(x_1\in[0,1]\times[0,1]\ni x_2) \land ((x_1,x_2)\in \mathbb{R}^2). \end{align} These $\ell$-areas intersect all other $\ell$-areas at some point; the intersection $A^1\cap A^2 \cap \cdots \cap A^{\ell} \not = \emptyset,$ is defined as the number string $(1000\cdots 0).$ The unit square  I$^2$ then becomes divided into $n=2^{\ell}$ disjoint unit areas $a_j \cap a_k = \emptyset$, where the $n$:th unit area $a_n$ is the complement unit-area, $a_n(000 \cdots 001)\not \subset A^{\ell} $, and thus have the unit-symbol $p.q.r.s.\cdots s_{\ell}.$ which is fully \emph{saturated}, just containing negated literals, denying any contact with the $\ell-$areas.  \end{definition}  For the binary operators, we represent the selectors $p_1(1100)$ and $q_1(1010)$ as two areas in I$^2$ with a non-void intersection and complements that do not contain each other. We let in the binary model case (the extended case) $$(\hat p \cap \hat q)\equiv(\hat p \land \hat q)\cong [\hat p  \hat q]; (1100 \cap 1010) =1000 $$ and we shall here identify the set-theoretic operations of $\cap$ and $\cup$ with the logical symbols $\land$ and $\lor$ as they give the same truth tables; see [11] for the details. There are $2^{2^\ell}$ members, called connectives or operators, of the type $op_{\boldsymbol{x}}$ in the operator space $\mathcal{OP}^{\ell}$ where $\ell$ thus is the number of letters $\ell=|p,q,r,s,p_1,q_1,r_1,s_1\cdots s_{\ell}|.$ \begin{definition}We put the topology of \emph{lexicographical} (dictionary) order[8], defined by \begin{align}((a,b)<(c,d)) \land  ((a=c) \land (b<d))\end{align} or for descending order as \begin{align}((a,b)>(c,d)) \land \ ((a=c) \land (b>d))\end{align}  on the unit square. We let the four place number string $x_1x_2x_3x_4$, or vector $(x_1, x_2, x_3, x_4)$, be ordered so that we have $x_1<x_2<x_3<x_4$.\end{definition}
 This means we begin the area notation with $x_1$, the first basis vector (1000), and by definition we get four disjoint unit areas, meaning their binary strings share no common 1s(in a normal space). We then choose the intersection $p\land q$ as a unit area 1, denoted $a_1$.  With this first choice $pq$ we follow the topology again to get $(2) p(q+1),(3)(p+1)q$ and $(4)(p+1)(q+1)$ now in descending order. If we compare $p(1100)$ and $q(1010)$, each number represents a unit area, $a_{n=1,2,3,4}$, with $p$ to the left of $q$, so it is clear which areas belong to $p$, to $q$, or to both. Each binary operator corresponds to one or several specific unit-areas; for example, the disjunction $\lor(pq+p+q)$ has the number string (1110), so it occupies every area except $a_4$, while the non-conjunction $\downarrow (0001), pq+1,$ occupies just one, $a_4$, the complement unit area. The conditionals are the strings $$\frac{\Rightarrow 1011}{\Leftarrow 1101}\mapsto pq+\frac{(p+1)q}{p(q+1)}+(p+1)(q+1)=\frac{pq+p+1}{pq+q+1}=\frac{p(q+1)+1=pq.'}{(p+1)q+1=p.q'}.$$  The computation in the vector case becomes
$$(1,0,0,0)+\frac{(0,0,1,0)}{(0,1,0,0)}+(0,0,0,1)=\frac{(1,0,1,1)}{(1,1,0,1)}.$$ We also use this vector view, but we call it adding and multiplying number strings, modulo 2, see Remark 2.5 below. \begin{definition} If $p=q$ or $(p\subset q) \ | \ (q\subset p)$, in the case of singular operators, we use the one-dimensional unit line I$=[0,1]$, to get two defining intervals, $p$ and $\lnot p,$ by instead considering the quotient topology I$^{2}/\sim$ = : I, identifying the points $(0,1)=(0,0)$ and $(1,1)=(1,0).$\end{definition}  \begin{remark} In the binary model case, we see this as a four-place coordinate ring of I$^2$, where we visualize the coordinates as ordered areas, which also can be seen as the basis vectors $(1,0,0,0), (0,1,0,0), (0,0,1,0), (0,0,0,1)$. We consider the other operators as linear combinations of the basis vectors. For example, the biconditional can be composed as $$ 1+4=(1,0,0,0)+(0,0,0,1)=(1,0,0,1). $$ In fact, the vector space and the unit-areas are isomorphic, but we prefer the geometric approach; it makes the other representations (operators) and different possibilities easier to present. \end{remark} \begin{definition} The PBNF-transform uses only polynomial multiplication and addition,  + and $\cdot$, modulo 2, defined by for the binary model case $$\mathcal{PB} \times \mathcal{PB} \to \mathcal{PB}:   \ [g{*}] \ni ( +: p_x+p_y=+p_x p_y) \ \text{and} \ (\cdot \ : p_x \cdot p_y=p_xp_y) \in [g{*}]) $$ $$  \ \sum^{2n}_1 p_{x_j}+ \ldots + p_{x_k}=0, j=k, \sum_1^{2n+1} p_{x_j}+ \ldots + p_{x_k}=p_{x_k}, j=k ; \  \ \pi_{x_1} (p_{x_1})^{\ell}=p_{x_1}. $$  We can regard the repeated multiplication as a projection $\pi: \mathbb{R}^{\ell} \to \mathbb{R}$ on the first coordinate $\pi(x_1, x_2\cdots, x_{\ell})=x_1$ numbering the ingoing terms in the exponent. Repeated addition is idempotent for an odd number of terms, but annihilates for an even number. We identify addition with subtraction in $\mathbb{Z}_2$. For the inverse of $p$, we use $p.$ or $(p+1)$; we can call this "conjugate" instead of "negate" the literal or statement. \end{definition}
The \emph{Selectors}, or inputs, which we index here as $p_j, q_k$, $j=1,2,3,4; k=0,1$, are commonly used only as $p$ and $q$, but here we can use all possibilities.
In the table below, we present the different inputs; we note that the diagonal and the anti-diagonal with 1s and 0s, which we cannot use.\begin{center}
\begin{tabular}
{| c| c | c | c | c || c | c |}
\hline
$\circ$ & $p_11100$ & $p_21001$ & $p_30110$ & $p_40011$ &
$q_11010$& $q_20101$\\
\hline
$p_11100$ & $0$ & $p_1p_2$ & $p_1p_3$ &$1$ & $p_1q_1$&$p_1q_2$  \\
\hline
$p_2$1001 & $p_2p_1$ & $0$ & $1$ & $p_2p_4$& $p_2q_1$&$p_2q_2$ \\
\hline
$p_3$0110 & $p_3p_1$ & $1$ & $0$& $p_3p_4$& $p_3q_1$& $p_3q_2$ \\
\hline
$p_4$0011 &$1$ &$p_4p_2$  & $p_4p_3$ & $0$& $p_4q_1$& $p_4q_2$ \\
\hline
\hline
$q_1$1010& $q_1p_1$ & $q_1p_2$ & $q_1p_3$& $q_1p_4$&$0$ & $1$\\
\hline
$q_2$0101& $q_2p_1$& $q_2p_2$ & $q_2p_3$ & $q_2p_4$&$1$&$0$ \\
\hline \end{tabular} \end{center} Note that $p$ has four faces $(1100-p_1, 1001- p_2, 0110-p_3, 0011-p_4,),$ but $q$ can have only two $(1010-q_1, 0101-q_2)$ . (We also use dictionary order here to get $p_1'=p_4$). Also note that the selectors do not commute; polynomial symbols commute $p(q+1)=(q+1)p$ regarded as real numbers, but not as selectors, because a selector-combination is considered as an ordered pair $\{p,q \}=\{\{p \},\{ pq \}\}$ again in dictionary order as a matrix, e.g., $\begin{smallmatrix} p& q \\1& 1 \\ 1 & 0 \\ 0& 1 \\ 0 & 0 \end{smallmatrix} \not = \begin{smallmatrix} q& p \\1& 1 \\ 0 & 1 \\ 1& 0 \\ 0 & 0 \end{smallmatrix}.$ \begin{lemma} Let the Polynomial House, $\mathcal{PB},$ contain polynomial forms, linear or quadratic, $\emph{symbols},$ in a uniqe combination for a $n-$ place binary string $x_1x_2  \cdots x_n$, each string representing one logical output in the two-valued Propositional Calculus. The  $n=2^{\ell}$ \emph{unit symbols} are connected to a unit area member, or a unit number string,  $\textbf{x}(x_1,x_2,\ldots,x_i, \ldots, x_{n})$, and $(x_i=1\land x_j=0, i\not =j)$ is the i:th of the $2^{\ell}$ unit areas in the unit square. The different logical outputs, or area combinations, are described by \emph{DBNF}, Disjunctive Boolean Normal Form, as a disjunction of phrases ($x_i\land x_j =x_i x_j$) in  $$(x_1 x_2  \cdots  x_{n})\lor (x_1 x_2  \cdots  x_{n}') \cdots\lor (x_1 x_2'  \cdots  x_{n}') \lor(x_1' x_2' \cdots x_{n}')=1$$ where $(\cdot)'$ indicates the complement.\end{lemma}\begin{proof} From the lower cases we have \begin{align} x(10) \lor x'(01)=(11)\\ x_1x_2(1000)\lor x_1x_2'(0100)\lor x_1'x_2(0010) \lor x_1'x_2'(0001)=(1111)\\ x_1x_2x_3(10000000)\lor x_1x_2x_3'(01000000)\lor \cdots \\ \lor x_1'x_2'x_3 \lor x_1'x_2x_3' \lor x_1x_2'x_3'\lor x_1'x_2'x_3'(00000001)=(11111111) \\ \nonumber \cdots \cdots \\x_1 x_2  \cdots x_{n}(10\cdots 00)\lor (x_1 x_2  \cdots x_{n}')(01\cdots 00)\lor \cdots \\ \lor x_1 x_2'  \cdots x_{n-1}' x_{n}'(00\cdots 10)\lor x_1' x_2'  \cdots x_{n-1}' x_{n}'(00\cdots 01)=(11\cdots 11)\end{align} so the $n-$case follows by induction. These unit areas are disjoint $a_j \cap a_k = \emptyset$ and their union $a_j \cup a_k=a_j + a_k$   is regarded as an addition of real numbers.  \end{proof} \begin{lemma}The \emph{DBNFs} transform into \emph{PBNF}, Polynomial Boolean Normal Form, by multiplication for conjunction and addition for disjunction, using $g(p,q,1)$, because all unit areas are disjoint. \begin{align}g((x_1 x_2  \cdots x_{n})\lor (x_1 x_2  \cdots x_{n}')\lor \cdots \lor (x_1 x_2'  \cdots x_{n-1}' x_{n}')\lor(x_1' x_2'  \cdots x_{n-1}' x_{n}'))\\ = p_1p_2\ldots p_{n} + p_1p_2\cdots p_{n}.+\cdots  +p_1p_2.\cdots p_{n-1}.p_{n}.+ p_1.p_2.\cdots +p_{n-1}.p_{n}. \end{align}\end{lemma} \begin{proof} We translated the DBNFs into multiplication and addition, which was possible for the unit areas because of the definitions $\land:=$ min$(p_j,q_k)$ and $ \lor:$=max$(p_j,q_k)$. The case $\lor:=$ max$(1,1)$ was non-existent. \end{proof}  \begin{definition} A Truth-Table is a tableau of $\ell$ letters and $2^{\ell}$ of lines of number strings. We always start with the full homogeneous case of ones, and end with the homogeneous case of zeros, as follows. Each line of numbers is called a numbered case: \emph{Case j}: $1\leq j\leq 2^{\ell}$
$$\begin{smallmatrix}  & x_1 & x_2 & x_3 & x_4 & x_5 & x_6 & x_7 & x_8 & x_9 & x_{10} & x_{11} & x_{12} & x_{13}& x_{14} & x_{15} & x_{16}& \cdot & \cdot & \cdot  & x_{\ell}& S_{c(k,\ell)} \\  Case \ 1 & \boldsymbol{1} & \boldsymbol{1} & 1 & 1 & 1 & 1 & 1 & 1 & 1 & 1 & 1 & 1 & 1 & 1 & 1 & 1& \cdot & \cdot & \cdot  & 1 & ? \\Case\ 2& \boldsymbol{1} & \boldsymbol{0} & 1 & 1 & 1 & 1 & 1 & 1 & 1 & 1 & 1 & 1 & 1 & 1 & 1 & 1 & \cdot & \cdot & \cdot& 1 & ?  \\  Case \ 3 & \boldsymbol{0} & \boldsymbol{1} & 1 & 1 & 1 & 1 & 1 & 1 & 1 & 1 & 1& 1 & 1 & 1 & 1& 1 & \cdot & \cdot & \cdot & 1 & ? \\Case \ 4 & \boldsymbol{0} & \boldsymbol{0} & 1 & 1 & 1 & 1 & 1 & 1 & 1 & 1 & 1 & 1 & 1 & 1 & 1 & 1 & \cdot & \cdot & \cdot &  1 & ?\\ & \cdot & \cdot & \cdot& \cdot & \cdot & \cdot & \cdot & \cdot & \cdot& \cdot & \cdot & \cdot &  \cdot & \cdot & \cdot& \cdot & \cdot & \cdot & \cdot & \cdot &  ?   \\Case\ j&\cdot & \cdot & \cdot&  \cdot & \cdot & \cdot& \cdot & \cdot & \cdot& \cdot & \cdot & \cdot & \cdot & \cdot & \cdot& \cdot & \cdot & \cdot & \cdot & \cdot & ?  \\ &\cdot & \cdot & \cdot&  \cdot & \cdot & \cdot& \cdot & \cdot & \cdot& \cdot & \cdot & \cdot & \cdot & \cdot & \cdot& \cdot & \cdot & \cdot & \cdot & \cdot & ? \\Case \ 2^{\ell}& 0 & 0 & 0 & 0 & 0 & 0 & 0 & 0 & 0 & 0 & 0 & 0 & 0 & 0 & 0 & 0& \cdot & \cdot & \cdot & 0 & ?\end{smallmatrix}$$ \end{definition} \begin{theorem}Every mode of composition, $S_{c(k,\ell)},$ describable by a truth table is translatable into terms of denial $(\lnot)$, conjunction $(\land)$, and alternation $(\lor).$ \end{theorem} \begin{proof} This is verbatim from Quine[9] who refers to Emil Post. Alternation is now often named disjunction. Case 1 from the definition above is a case of conjunction, i.e., every $x_j$ is true. We write this $\chi_1=x_1\land x_2 \land \ldots \land x_{\ell}$ is true. From this follows, for Case 2, that $\chi_2=x_1\land x_2' \land \ldots \land x_{\ell}$ is also true. Case 3 is $\chi_3=x_1'\land x_2 \land \ldots \land x_{\ell}$ is true  and Case 4 is $\chi_3=x_1'\land x_2' \land \ldots \land x_{\ell}$ is true. For Case $i>4$ it is the case when $\chi_i= x_1 \land x_2 \land \cdots \land x_i' \land \cdots \land x_{\ell}$ is true. This means every case between $1$ and $2^{\ell}$ is a conjunction of literals. Now every truth table is part of the table above for the appropriate $\ell$ (here we set $2^2=4$ by choosing $\chi$ as we did above). For the question marks, there are two cases: if only one 1 (T) is needed to get 1, then it must be a conjunction. Case 2 means that multiple cases form the question mark, so it must be a disjunction of, say, $\chi_i \lor \chi_j \lor \cdots \lor \chi_k$, which is true if any of the phrases are true. This means we can form all cases using $\lor$, $\land$, and $\lnot$ as desired. The trivial case of all zeros can be assigned the contradictory case. The conclusion of the proof is: \emph{All truth functional modes are translatable into terms of disjunction, conjunction, and negation. We can take these as the sole basic or primitive modes of composition and construct all others merely as results of repeated application of these primitive modes.} \end{proof}
What is missing now is the transformation from the logical outputs in PC to polynomial symbols using PBNF $$\land \rightarrow pq,\lor \rightarrow (p+1)(q+1)+1, \lnot \rightarrow p+1.$$ We could take the statement as obvious, but we shall discuss the details anyway, by connecting to the comment Quine gives for the example I$^2$ where we added the polynomial symbols
\begin{corollary} The connections between DNF and PBNF are the following\begin{align}
i) \chi_1=x_1\land x_2 \leftrightarrow pq\\
ii) \chi_2=x_1\land x_2' \leftrightarrow p(q+1)\\
iii)\chi_3=x_1'\land x_2 \leftrightarrow(p+1)q\\
iv)\chi_4=x_1'\land x_2' \leftrightarrow(p+1)(q+1)
\end{align} \end{corollary} \begin{proof} We start the proof with an example that follows Quine's comments. For the conditional, which has 1 in all places but the second, we get $\chi_1\lor \chi_3 \lor \chi_4$ and $$(x_1 \land x_2) \lor (x_1'\land x_2) \lor (x_1'\land x_2').$$ If we put this as an equation, and use that it is the complement of area 2, and that the conditional also have $(x_1' \lor x_2) \in$ DBNF, and by de Morgan $((x_1\land x_2'))'=  (x_1' \lor x_2)$ \begin{align}((x_1\land x_2'))'=(x_1 \land x_2) \lor (x_1'\land x_2) \lor (x_1'\land x_2')=( x_1' \lor x_2) \in \text{DNF} \end{align} We now assume that we can write for (2.34)$$\cong p(q+1)+1=pq+(p+1)q+(p+1)(q+1)=(p+1)q+(p+1)+q.$$  We defined  $x_1=1000 +0100=1100:= p$, and  $x_2=1000+0010= 1010:=q$ for the lexicographical order and furhter defined conjunction as min(1|0) and disjunction as max(1|0) \begin{align}
i) x_1\land x_2 \cong \text{min}({1100},{1010})=1000  \\
ii) 0\\
iii)x_1'\land x_2 \cong \text{min}({0011},{1010}) =0010 \\
iv)\chi_4=x_1'\land x_2'\cong \text{min} ({0011},{0101})=0001. \end{align} These areas are disjoint, so we can add the number strings $1000+0000+0010+0001=1011$. We can now conclude the following about the unit areas: \begin{align}
i) pq=(1100)(1010)= 1000\cong (x_1\land x_2)\\
ii) p(q+1)= 1100(1010+1111)=1100\cdot 0101=0100 \cong (x_1'\land x_2)\\
iii)(p+1)q=1100+1111(1010)=0011\cdot 1010 \cong (x_1\land x_2')\\ iv)(p+1)(q+1)=0011\cdot 0101=0001 \cong (x_1'\land x_2').
\end{align} We have identified the polynomial symbols and the number strings, following Remark 2.5. The step between the number string and the logical symbols is a congruens modulo the polynomial bundles $[g_*]$ in (2.1). Our assumption above is thus correct, we need just fill in number strings using the last equations, and the example we took could be substituted for any other operator.\end{proof}  Now we prove the extension of the transfer function $g$ in (2.6) to an $\ell \in \mathbb{N}$-divided unit square. The following iterative process suits higher arity.  \begin{definition}The Boolean transfer function $g$: Op $\to \mathcal{PB}$ defined in (2.6) is linear for $S_{c(k,\ell)}$ (the statement form, $c$=classical, $k$=counting the operators(connectives), $\ell$=total of literals) the following iterated way. \begin{align} g(S_{c(k,\ell)}  * S_{c(k,\ell)})=g(S_{c(k,\ell)} \ g (*) \ g(S_{c(k,\ell)})) \\ =(s^{j-1}_{g} \ g(*) \ s^{j}_{g}) =(p,q,a)\circ (p,q,a) = [(\hat p+1)(\hat q+1)]\\ \hookrightarrow S_{c(k,\ell)}; \ a \land a_{j=1,2,3}=0|1. \end{align} \end{definition}  \begin{theorem} We can also generalize to a fix but unbounded $\ell$ by instead considering the series of compositions $\circ$ by $g_{\kappa}$ \begin{align} \sum_{i=1 \cdots n} (\sum_{\kappa} g_{\kappa}(S^{j-1}_{c(k,\ell \geq 3)} \  * \  S^{j}_{c(k,\ell \geq 3)}))=\sum_{i=1 \cdots n}  \hat s_g^{j-1}  \hat s_{g}^{j}\end{align} where $\kappa$ gives the number of forms composed by $g_{\kappa} $ at a time. \end{theorem} \begin{proof} The \emph{binary case}, which we refer to as our \emph{model}-case, because we can, in fact, use the model case for all $ \ell$-ary cases, $\ell>2 \in \mathbb{N}$. This rests on the fact that any domain of letters and connectives can be described in a truth table using $\land, \lor$ and $\lnot$ (that is, in DBNF), and that DBNF can be put into PBNF.
Any operation in an $\ell$-divided unit square  $g(p_j (*) q_k)$ forms a \emph{binary} iteration process; the letters involved are just a renaming of the original letters $(p,q)$, such that, for example, $\hat p \mapsto (r \ldots r_{\ell})$ and in the same way $\hat q \mapsto (s \ldots s_{\ell} )$. For higher arity, we get $2^{\ell}$ unit areas, but as we commented before, we handle this $2^{\ell}$-arity in model-binary form. The last step in this composition then contains $\ell$ variables, but since these monomials $p(p,q,r\ldots)$ cancel the same way modulo 2 ($\hat p+ \hat p=0$) and $p^{\ell}=p$, this works.\end{proof} \section{Applications to Statements of Higher Arity - SAT problems} We start this section by calculating problems with $\ell=2,3 $, before we let $\ell<C $, fixed but not bounded in the SAT problems. The initial examples are meant to recall the method, introduce new readers to it, and as a help to get accustomed to the new compact notation for the polynomial symbols. \begin{example}Let the conjunctive compound statement (and recall that we also can use dots in the operator space, now as inner parentheses) $$S_{c(3,2)}=  (.p\Rightarrow q \land .q\Rightarrow p) $$ that has two conditional sub-statements  $S_{c(1,2)}^1$ and  $S_{c(1,2)}^2$ be written, following the table of symbols, as  $$ g(S_{c(3,2)})=(pq.).(p.q)'=pq.+p.q+1=p+q+1$$  using $(p.q)(p.q)=0$ and this we shall use over and over again, as it permits "multiplication in and addition out" of the clauses in SAT. We shall also use these contraction rules $p.q+q.p=p+q$ in $g(p,q,1),$ for more see Lemma 3.20. To compare we find $$((p.q).(pq.).)'=((p.q)+1)((pq.)+1)'=p+q+1 \in g(p',q',0).$$
The inverse function (the fiber) $g^{-1}:\mathcal{PB} \to$ Op gives: $g^{-1}_{(p,q,1)}(p+q+1)= p\Leftrightarrow q$ and $g^{-1}_{(p',q',0)}(p+q)= p\Leftrightarrow q.$  \end{example}
\begin{example} The next example is presented directly in PBNF by the family $g(p',q',0)$ $$((a.b.).).c+a.(b.c), $$ and when we use the fiber $(\hookrightarrow)$ to the logical symbols, we get for the equivalence $g(p',q',0)\ni (p+q) \hookrightarrow(p \Leftrightarrow q)  \in\text{Op}_{(..)}'$, the first term is a conjunction in a conditional, the last are two consecutive conjunctions, by the symbol table. We get the logical representation  $$\hookrightarrow  (.A\land B \Rightarrow C)\Leftrightarrow (A \Rightarrow .B \Rightarrow C) \in\text{Op}_{(..)}'.$$
The polynomial can be reduced to $a.b.c+a.b.c=0$ by canceling complements. In this example the logical symbols gave a better view, but they gave no clue how to proceed with a proof. If we where not given the family, we can start with $(a.b.).$ which is either $\lor$ or $\land$. If we try $\lor$ then $((a.b.).).c$ must be $(A\lor B) \not \Leftarrow C$, and $a.(b.c)$ must be $A \not \Leftarrow (B \not \Leftarrow C)$ in all $$(A\lor B) \not \Leftarrow C \not \equiv A \not \Leftarrow (B \not \Leftarrow C).$$ This is the dualization proposed by Church in [4, 16, Duality] and here in (2.8).\end{example} \begin{lemma} The distribute law $ p(q.r.)' = (pq).(pr).'$  can be proved in $g(p,q,1).$  \end{lemma} \begin{proof} We find $(pq).(pr).'=pqr+pq+pr=p(qr+q+r)=p(q.r.)'$, by ordinary calculation of real numbers. \end{proof} \begin{lemma} We also prove transitivity under the conditional $$(.p\Rightarrow q \land .q \Rightarrow r ) \Rightarrow .p \Rightarrow r$$ \end{lemma}\begin{proof} Here we give the logical view because we change it to $$(1)\lnot(.\lnot p\lor q \land .\lnot q \lor r)  \lor (\lnot p \lor r)=(.p \land \lnot q \lor .q  \land \lnot r) \lor (\lnot p \lor r)(2)$$
by using $p\Rightarrow q= \lnot p\lor q$ four times. This is because it now looks more like a SAT problem. We map this into the compact form of $g(p',q',0)$ $$(1)((p.q).(q.r).)..(p.r)=((p.q)+(q.r)+1)(p.r)=p.qr+p.q.r+p.r=p.r(q+q.+1)=0$$ $$(2)((p.q).(q.r).)(p.r)=((p.q)+(q.r)+1)(p.r)=0$$
 For the two dots standing together, one comes from $(p.q.).$, one from the outer negation. We also note that the composition $(p.q).(q.r).'=(p.q)+(q.r)$ becomes addition, due to the change of parity in $q$, which we shall use soon. Sometimes it is conviniente to rewrite the logical form, but the gain was her minor, if we transfer directly we get $ ((p.q).(q.r).)..(p.r)$ identically with (1). The reason for this are that the polynomial bundles $[g_*]$ are equivalence classes as can be seen in the table. \end{proof}  Now to the SAT problems, that is, a Boolean satisfiable problem, with clauses $S^j_{c(\lor,\ell\geq 3)}$ binded by conjunctions of the type \begin{align} S^1_{c(\lor,\ell \geq 3)} \land S^2_{c(\lor,\ell\geq 3)} \land \ldots \land S^{j-1}_{c(\lor,\ell\geq 3)} \land S^j_{c(\lor,\ell\geq 3)}\land\ldots \land S^{n-1}_{c(\lor,\ell\geq 3)} \land S^n_{c(\lor,\ell\geq 3}\\  \land  \ S^{n}_{c(\lor,\ell\geq 3)} \land S^{n-1}_{c(\lor,\ell\geq 3)}\land \ldots \land S^j_{c(\lor,\ell\geq 3)} \land S^{j-1}_{c(\lor,\ell\geq 3} \land \ldots S^2_{c(\lor,\ell\geq 3)}\land S^1_{c(\lor,\ell\geq 3)}. \end{align} We are here supposed to find conditions that make $\hat S_{c(\land(\lor),\ell \geq 3)}$(short for (3.1) and (3.2) by the hat) true, meaning some permutation $1|0$ of the truth values of the ingoing literals $(p,q,r,s\cdots s_\ell)$ that satisfy the statement; then we say that we have a YES instance to the SAT. If that is not possible, we say we have an UNSAT. If we transform (3.1) and (3.2) with PBNF into a Polynomial Boolean Equation PBE= $\sum_{j} s^{j}_g=0$, in a complement family, the question is to explore \begin{equation} \text{ker} (\sum_{j} s^{j}_g)  \equiv (\sum_{j} s^{j}_g=0))^{-1} \ \text{or just} \ \text{ker} \hat s_g=0)\end{equation}
The kernel is also called the null space, and we can write $$\sum_{j} s^{j}_g=(p-a)(q-a)(r-a)(s-a)\cdots (s_{n}-a)=0 ; \ a=1|0$$ for some combination of parentheses, and this expression is zero if we just find one value that is correct for the polynomial, which is equivalent to a YES instance to the SAT.    
For this transformation, we start by choosing the family $g(p',q',0)$ to help us, as here the composition is just multiplication for the multinomials $\hat p \hat q$, the simplest we can find for the clauses. Other choices would require composing the literals in the clauses, which would be impractical, since there are far more literals than conjunctions. This choice means that we get the more complicated composition $\hat p.\hat q.'$ for the conjunctions, but we shall see that there often will be simplifications so that $\hat p.\hat q.'=\hat p + \hat q.$ If not we can also just ignore $\hat p\hat q$ as we still are able to get some solutions from the additiv ker $\hat p + \hat q$, as we discussed in the introduction.  \begin{remark} We could of course get the factorization as final composition if we instead use a normal space $g(p,q,1)$, as the composition $\hat p \hat q $ for the conjunctions is already in factor form. In [10], the authors discuss improving the connection between CNF and algebraic polynomial form by using a \emph{power term polynomial algebra} as an abstraction language. This is motivated by a \emph{tiling mismatch} that arises when the CNF is broken into pieces to facilitate conversion, and by the resulting polynomial pieces. The method suggested for avoiding this mismatch builds on a powerset construction $\mathcal{P}(p)\setminus\emptyset$ of the polynomial; for example, $p(x_1,x_2)=x_1+x_2 + x_1x_2$ is $$\mathcal{P}(p)\setminus\emptyset= \{\{x_1 \}, \{x_2 \}, \{x_1x_2 \}\}$$ or just $\mathbf{P}_{1,2}$. Another letter $S$ is added to the expression to indicate repetition of the letter in all clauses; none here, so $S_{\emptyset} \mathbf{P}_{1,2}$ for $x_1+x_2 + x_1x_2$ is called the power term. The system has some assemblies in structure with PBNF, but with much heavier overhead. In the next example the system needs eight lines to convert, choosing the composition carefully to use reversed repetition of literals $\chi_1 \circ \chi_3, \chi_4 \circ \chi_6$ and $\chi_5 \circ \chi_7$ leaving $\chi_2$ \begin{align}\phi :=(x_1 \lor  x_2) \land (x_2 \lor x_3) \land  (\lnot x_1 \lor \lnot x_4 \lor x_5) \land  (\lnot x_1 \lor x_4 \lor x_6) \land \\ \nonumber (\lnot x_1 \lor \lnot x_5 \lor x_6) \land (\lnot x_1 \lor x_4 \lor \lnot x_6) \land  (\lnot x_1 \lor \lnot x_5 \lor \lnot x_6)\end{align} ending in $S_2 \mathbf{P}_{\emptyset}+S_{1,2}\mathbf{P}_{\emptyset}  \mapsto x_1x_2+x_2$, and the solution $x_1.x_2=1$ or $0,1,x_3,x_4,x_5,x_6.$ This is a YES instance, an obvious choice as $\lnot x_1$ is a member of sex clauses, and the other two contain $x_2.$
Here we argue that the "mismatch" is more about how we address the problem; it seems easier to dualize first and then compose the multiplied clauses (phrases). We get after polynomial transformation, with the direct approach first, where $\hat r=r_1 \circ r_2 \cdots \circ r_n$ and $\hat s=s_1 \circ s_2 \cdots \circ s_n,$ with the composition $\circ$ first as $ rs+r+s$ and then as $\hat p \hat q$ \begin{align} \hat S_{c(\land(\lor),\ell \geq 3)} \mapsto \hat p \hat q(\hat r\hat s+\hat r +\hat s) = p(\hat r\hat s+\hat r+\hat s)q(\hat r\hat s+\hat r + \hat s)=1\end{align} We got the factorization but the cost was the number of terms equalizing 1 are many more than in the following where we used the dualization first and now $\hat r=r_1 r_2 \cdots r_n$ and $\hat s=s_1 s_2 \cdots s_n,$ with the first composition $\circ$ as $ rs$ and then $p.q.'$ or $pq+p+q$\begin{align} \text{Multiplicativ way}\ \hat S_{c(\land(\lor),\ell \geq 3)} \mapsto \hat p. \hat q.'(\hat r\hat s)= (p(\hat r). q(\hat s).)'=0\\ \text{Multiplicative additiv}\ (\hat p \hat q+\hat p+ \hat q)(\hat r \hat s)  =p(\hat r)q(\hat s)+p(\hat r)+q(\hat s)=0 \\ \text{Pure additiv}\ \sim p(\hat r)+q(\hat s)=0 \ \text{by Theorem 3.15}.\end{align}  We also obtained the equation in normal form, and as we will show, we can often simplify and factorize it to find a solution.  Notably, (3.6) is more economical: we substitute $(\hat r, \hat s)$ in two places, whereas (3.7) has them in four, and (3.8) returns to two.  We also note that we \emph{can} make different conversion choices by using one of many polynomial families; this flexibility is an important feature of the PBNF-transform. \end{remark} \begin{example}The following is the representation of a SAT-problem written with $g(p',q',0)$, where the clauses become phrases with multiplied literals and $\circ$ is the conjunctive composition. \begin{align} (x_1  x_2 x_3 \circ  x_1. x_2  x_3) \circ (x_1  x_2. x_3 \circ x_1  x_2   x_3.)\end{align} We shall use this as an example when we introduce the steps for conversion. The form above is not fully completed as a polynomial equation; the first step is done, but the conjunction compositions remain before we can equate the form to zero to complete the transfer to a Polynomial Boolean Equation (PBE=0). \end{example} The next step is to choose a model for the compositions for the conjunctions; we can use them differently; $ pq+ p+ q$ describes the composition as the sum of areas 1, 2, and 3, and this can be written as $pq+p(q+1)+(p+1) q=pq+pq.+p.q$ or $ pq.+q$ or $p.q+p$. Another way to look at this is to form the complement, $ p. q.',$ of area 4; in this way, we have not multiplied $pq$, but instead $p.q.$ hoping for annihilations $p.p=0$, to reduce the equation, and we also just substitute in two terms. \begin{equation}\left\{ \begin{array}{ll}
\text{Conjunctive inner additive composition}\\
\text{without reversed repetition in} \ g(p',q',0)\\
1. (\sum_i s_g^{i})(\sum_j s^{j}_g)+\sum_{i} s_g^{i}+\sum_j s^{j}_g=0 \\
2. (\sum_i s_g^{i})(\sum_j s_g^{j})+(\sum_i s_g^{i})(\sum_j s_g^{j}.)+(\sum_i s^{i}_g.)(\sum_j s_g^{j})=0 \\
3. (\sum_i s_g^{i})(\sum_j s_g^{j}.)+\sum_{j} s^{j}=0\\
4. (\sum_i s^{i}_g.)(\sum_j s_g^{j})+\sum_{i} s_g^{i}=0  \\
 \text{Conjunctive outer multiplicative composition in} \ g(p',q',0)  \\
5. (\sum_i s^{i}_g.)(\sum_j s_g^{j}.)'=0 \ \equiv \ \sum_i s^{i}_g.\sum_j s_g^{j}.=1\end{array}\right. \end{equation}
The following are equivalent formulations of additive phrase composition, providing step-by-step reducible equations.\begin{equation}\left\{ \begin{array}{ll}
1. \ s^{i}_g \circ s^{j}_g:= s^{i}_gs^{j}_g + s^{i}_g +s^{j}_g\\
2. \ s^{i}_g \circ s^{j}_g:= s^{i}_gs^{j}_g. + s^{j}\\
3. \ s^{i}_g \circ s^{j}_g:= s^{i}_g.s^{j}_g + s^{i}\\
4. \ s^{i}_g \circ s^{j}_g:= s^{j}_g + s^{i} ;  \ s^{i}_gs^{j}_g=0 \ \lor \ \text{ker} \hat s_g \setminus \text{ker}s_g^i s_g^j
\end{array}\right. \end{equation}  
The formulas above are the same when the multiplicative term is zero,
which happens in (3.8) as we have reversed repetition of literals, as in $x_1x_1.=0$ and $x_2x_2.=0$ and so on.
For our example, we take number 1 because we have symmetric reversed repetition of literals, and now we get a Polynomial Boolean equation of degree 3
\begin{align} x_1  x_2 x_3 +  x_1. x_2  x_3 + x_1  x_2. x_3 + x_1  x_2   x_3.=0 \end{align} but in fact it is only of degree 2, due to cancellations of higher order terms, $x_i. x_j  x_k=(x_i+1) x_j  x_k=x_i x_jx_k+x_jx_k$ four times for $i,j,k=1,2,3,$ so we got (PBE$^3)\cong$ (PBE$^2$)\begin{align}  x_2  x_3 + x_1 x_3 + x_1  x_2 =0. \end{align}  For the solution, we look for factorization. If we have real repetition, we can factorize it so that it can be written $$p_{j}( \hat s_g^{i}+ \hat s^{j}_g)=0.$$
This means that we decided the first solutions string in ker$( \hat s_g^{i}+ \hat s^{j}_g)$ which is $$ \text{ker}(\hat s_g)=\{ (\hat p_1, \hat p_2\cdots p_{j-1}, 0, \hat p_{j+1} \cdots  \hat p_{\ell-1}, \hat p_{\ell})\cdots(\hat p_1, \hat p_2\cdots p_{j} \cdots  \hat p_{\ell-1}, \hat p_{\ell})  \}$$  where $p_j=0$ and the other terms are arbitrary, the other strings are when $p_j\not= 0$. \begin{lemma}If $\hat s_g$ is repetitive in the clauses, then it has at least the solution string containing $p_j=0$, and other literals arbitrary, since $p_j\hat s_g=0$ by \emph{outside} factorization.\end{lemma} This often happens in simple SAT; for example, $pq.r+pqr.=p(q.r+qr.)=p(q+r)=0$ gives $(0qr),(1,q,q).$ In more complex problems, we can use direct factorization by equating parts of the equation in the final stage.
For our example, we find no solution this way, but instead, if we start with (3.12) \begin{align} x_2 x_3(x_1+x_1.)+x_1(x_2. x_3+x_2 x_3.) \\
= x_1x_2+x_1x_3+ x_2 x_3 =0  \land (000, 001, 010, 100).\end{align} We get the same result of lowering the degree by factorization as in (3.13) by eliminating the negated terms.
We can also \emph{conjugate away} lower-order terms by multiplying with a negated reversed letter, if we have repetition in all phrases $$r.(pqr+ p.q.r+pq.r+ p.q.r.)=0; \ p.q.r.=0$$ $r \not=1  \land 110.$
Another way to find a solution is by arranging factorization, which is possible because adding terms with reversed repetitions annihilates them, reducing the degree in one step. Also, adding higher terms to lower terms contracts to saturated phrases, where all literals are reversed. \begin{example} All literals are here different or alike but in different parity as in $$x_1  x_2 x_3 + x_1.  x_2.x_3.+  x_4. x_5   + x_4 x_5.=0 $$ so there is no repetition but we can then use additive annihilations when we add terms with reversed repetition as in $$x_1 x_2 x_3+ x_1.x_2.x_3.=(x_1x_2+x_1+x_2)+x_1x_3+x_2x_3+x_3+1$$ $$=x_1.x_2.+(x_1+x_2.)x_3+1$$ where we got a saturated term $x_1.x_2.$ from the first parentheses by factorizations.
For the last two terms $x_4. x_5 + x_4 x_5 = x_4 + x_5$ we get $$x_1.x_2.+(x_1+x_2.)x_3+x_4+x_5=1$$ to arrange a factorization that was not visible from the start, and find a solution 11001. \end{example} \begin{example} For $$(X_1\lor X_2.\lor X_1)(X_1.\lor X_1.\lor X_2.)(X_2 \lor X_2 \lor X_2) \in \text{Op}$$
 $$x_1x_2.x_1+x_1.x_1.x_2.+x_2x_2x_2=0\in g(p', q, 0)$$ that simplifies to
$x_1x_2.+x_1.x_2.+x_2=x_2.+ x_2=1 \ \equiv 1=0,$ thus ker $\hat s_g =\emptyset$. If we change to $x_2.+ x_2.=0$ we get a non-void ker $\hat s_g= \{00,01,10,11 \}$ and we see that $$\text{ker} \hat s_g= \text{ker}  (\hat s^i_g+\hat s^j_g) \cup \text{ker} \hat s^i_g \hat s^j_g =\{11,10 \} \cup \{00,01 \}.$$ \end{example} We can also use the following fact \begin{lemma} For the family $g(p',q',0)$, addition of terms means an equivalence $\hat p + \hat q \hookrightarrow \hat p \equiv \hat q,$ in the operator space. This means that every clause or term must have the same value, $1$ or $0$, so that $\hat p+ \hat q=1+1$ or $\hat p+ \hat q=0+0$.\end{lemma}In the next example, we consider the satisfiability in $S_{c(11,4)}$ \begin{example} $$(p\lor q\lor r )\land (\lnot p \lor q \lor \lnot r)\dot \land(p \lor \lnot q \lor s)\land (\lnot p \lor \lnot r\lor \lnot s)$$ is a YES instance to 3-SAT [5], and we can use a dot over a symbol to indicate principality.\begin{proof} Here we take some time to analyze the table below \begin{scriptsize} \begin{center} \begin{tabular}
{|c | c| c | c | c | c | c | c | c | c | c | c | c | c | c || c | c |}
\hline
 $ clause \ 1 $ & $*$   &$ clause \ 2 $ & $ *  $
& $ clause \ 3 $ & $*$  & $ clause \ 4 $ & $*$& &$*$& 1|0 \\
\hline
 $ p\lor q\lor r $ & $\land$  & $ \lnot p \lor q \lor \lnot r $ &$\land $  & $p \lor \lnot q \lor s$& $\land$ & $\lnot p \lor \lnot r\lor \lnot s $ &  & &&\\
\hline
$pqr$ & $\circ$  & $p.qr.$ &$\circ $  &$pq.s$
&$\circ$ & $p.r.s. $&  & \emph{added phrases} & &\\
\hline
$pqr$&$+$ &  $p.qr.$  & $\circ$ & $pq.s$ & $+$ &$p.r.s.$ &  & & &\\
\hline
 $pqr$&+&$p.qr.$& +&$pq.s$&+ &$p.r.s.$ &+ &$p.r.s.q$ &  = & 0\\
\hline
$pqr$&+&$p.qr.$& +&$pq.s$&+ &$-$ &+ &$p.r.q.s.$ & =& 0 \\
\hline
\hline
\end{tabular} \end{center}\end{scriptsize}
$$\begin{smallmatrix} \boldsymbol{x} & (1 & 2 & 3) & (4& 5 & 6)\\
letters  & (p & q & r) & (\lnot p & q & \lnot r)\\ \end{smallmatrix}
\begin{smallmatrix} \boldsymbol{y} & (6 & 5 & 4) & (3& 2 & 1)\\
letters  &( p & \lnot q & s) & (\lnot p & \lnot r & \lnot s)\\
\end{smallmatrix}$$ We got higher-order terms in the final composition, but reductions may occur through factorizations, such as $p.r.s.(q+1)=p.r.s.q.$, giving the higher-order term the same parity, so it is saturated. This means that if it is equal to one, then $p=q=r=s=0$, but that gives a contradiction of $1=0$, because the other terms becomes zero as they are not saturated, thus $p.r.s.q.=0.$  Then Lemma 3.10 says that the other terms must have the same parity pairwise, which gives
\begin{equation}\left\{ \begin{array}{ll}
1. \ p.r.q.s.=pqr \ 1100\\
2. \ p.qr.=pq.s  \ 1100\\
\end{array}\right. \end{equation} Lots of other possibilities exist for satisfiability for this problem; in fact, it is easier to collect those that do not work, and these are the two trivial solutions 0000 and 1111, which fail in 1, together with 1110.  In 2, 1011, 0100, 0101, and 0100 fail. $$\text{ker}(\hat s_g)=\{ \hat s_g=0 | 0001, 0010, 0011, 0110, 0111, 1000, 1010, 1100, 1101 \}.$$  If we conjugate with $q.\not=0$ we get $p.r.q.s. +pq.s=0$ then we must have $011s$, so we cannot decide the whole kernel this way, only the members of the type above $(0110, 0111)$, but that is enough for a YES instance to SAT. Another view is that we can ignore a fully saturated phrase, since it must be zero; if there is more than one, they collapse to one or zero. If they should not be the same as $p.q.r.s.\cdots s_{\ell-1}. + p.q.r.s.\cdots s_{\ell}.$ we could contract them, but the result is no longer saturated, so we cannot ignore this term. For the example, this means that we just need $pqr+p.qr. +pq.s=0$, $$\equiv (pq+qr+q=pq.s)\land (q=0,p\lor s =0) \lor (p=1|0, q=1, r=0, s=s)$$
We get the following  ker(PBE$=0)\setminus (0000)=\{ (00rs), (p0r0),(1 1 0 s), (0 1 1 s)  \}.$  Now the trivial solution (0000) is in the kernel, since $(00rs)$ could be (0000), which contradicts the statement.  But the trivial solution would not be present if the saturated term were present, so when we ignore the fully saturated term, we must also ignore the trivial solution $(0000)$.
Another way to look at this problem is to find the \emph{exclusive solution}. The difference between inclusive or(OR) and exclusive or(XOR) is just the $pq$. Exclusive solutions in $g(p,'q',0)$ are $1qrs$ and ker$_{\setminus (0\cdots)}:=(pqr+p.qr.+pq.s +p.r.s.) \setminus (0qrs).$ In this case we just get $$p(qr+q.s)+ p.r.(q+s.)=0 \Leftrightarrow  (qr=q.s) \land p=1$$ and as we assumed $p=1$ we find  $(q=1,r=0,s=s \lor q=0, r=r, s=0)$ which give the following line $$\text{ker}(pqr+p.qr.+pq.s +p.r.s.) \setminus (0qrs) \in g(p,'q',0)=\{1100, 1101, 1000, 1010\}$$ $$\subset \text{ker}\hat s_g= \{ 0001, 0010, 0011, 0110, 0111, 1000, 1010, 1100, 1101 \}.$$
As we can see, the exclusive solution works; all values starting with 1, that we found in the exclusive solution, are in $\text{ker}\hat s_g.$\end{proof} Now we make a change in the second clause, so that the higher-order terms vanish \begin{scriptsize} \begin{center} \begin{tabular}
{|c | c| c | c | c | c | c | c | c | c | c | c | c | c | c || c | c |}
\hline
 $ clause \ 1 $ & $*$   &$ clause \ 2 $ & $ *  $
& $ clause \ 3 $ & $*$  & $ clause \ 4 $ & $*$& 1|0 &   \\
\hline
 $ p \lor  q \lor r$ & $\land$  & $ p \lor q \lor \lnot r $ &$\land $  & $p \lor \lnot q \lor s$& $\land$ & $\lnot p \lor \lnot r\lor \lnot s $ & & & \\
\hline
$pqr$ & $\circ$  & $pqr.$ &$\circ $  &$pq.s$
&$\circ$ & $p.r.s. $& & & added phrases   \\
\hline
$pqr$&$+$ &  $pqr.$  & $\circ$ & $pq.s$ & $+$ &$p.r.s.$ & & & \\
\hline
 $pqr$&+&$pqr.$& +&$pq.s$&+ &$p.r.s.$ &=&0 &     \\
\hline
&-&$pq$& +&$pq.s$&+ &$p.r.s.$ &=  & 0&   \\
\hline
\hline
\end{tabular} \end{center}\end{scriptsize}
As the table shows, we get a late reduction: $pqr+pqr.=pq(r+r.)= pq.$ We can see that $p.r.s.=1$  gives a contradiction again, so we cannot have $p=r=s=0$\begin{equation}\left\{ \begin{array}{ll}
1. \ p.r.s.=pq \ \land 1000\\
2. \ pq.s=0  \ \land 1000\\
\end{array}\right. \end{equation} \begin{align} \text{ker}\hat s_g= \{(0001, 0011, 0101, 0110, 0111, 1000, 1010\} \end{align} Here we see the difference, this is not a exclusive solution, the kernel contains all members, starting with 0 and 1.
This is suitable for manual computation; otherwise one can consider the polynomial $\hat p. \hat q.'$, which gives \begin{align} (pqr \ \circ \ pqr.) \ \dot\circ \ (pq.s \ \circ \ p.r. s.)=((pqr). (pqr.).)'\ \dot\circ \ ((pq.s).  (p. r. s.).)'\\ ((((pqr). (pqr.).).). (((pq.s).  (p. r. s.).).).)'=0 \end{align} \end{example} Here we do not make the reductions done in the table above, but they are clearly visible when we multiply $(pqr) (pqr.)=0$ and  $(pq.s)  (p. r. s.)=0.$
\begin{remark}The same idea as in Remark 3.5 of improving the direct correspondence between CNF and a polynomial representation is also discussed in [6]. The authors first demonstrate that the standard conversion yields many high-degree polynomials. They propose in the following a "bookkeeping" approach to exploit repetition (overlapping) in clauses to reduce these consequences of the polynomial conversion, by building blocks of clauses, trying to reduce to a special case, for example, where the literals in the clauses are the same but the repetition is odd, that is, either one or three are negated. The claim is: The logical representante of the polynomial $x_1+x_2+x_3$ is $$(\bar X_1 \lor X_2 \lor X_3) \land  (X_1 \lor \bar X_2 \lor X_3) \land (X_1 \lor  X_2 \lor \bar X_3)\land  (\bar X_1 \lor \bar X_2 \lor \bar X_3).$$
\begin{proof} We transfer this to $$ x_1.x_2x_3 + x_1x_2.x_3 + x_1x_2x_3. + x_1.x_2.x_3. =0$$ by Theorem 3.15 in $g(p',q',0)$ and we lower the degree by Lemma 3.20 $$x_1x_2+x_1x_3+x_2x_3+ x_1x_2 +x_1x_3+x_2x_3+ x_1+x_2+x_3=x_1+x_2+x_3,$$ by also expanding $x_1.x_2.x_3.$ \end{proof} This is like Example 3.18 in [2], where we studied $+pqr$, the \emph{ternary parity} operator; it returns 1 when an odd number of inputs are 1. For a binary case it is the operator XOR:= |, with inputs $^{10}_{01}$, and symbol $p+q$. The appropriate inputs in the ternary case are 100, 010, 001, 111, and we get $p.qr+pq.r+pqr.+p.q.r.=p+q+r$, the same expression as above. To this the authors present as a proposition $$\phi=  (\phi \lor \psi)\land (\phi \lor \lnot\psi) \mapsto_{g(p',q',0)} (\phi \psi).(\phi \psi.).'=\phi\psi+\phi\psi. =\phi ,$$  which is (3.70) in Lemma 3.20. We use this for $\phi =\bar X_1 \lor X_2$ and $\psi=X_3$ to create the missing term $\bar X_1\lor X_2\lor X_3$, yielding a complete ternary operator equation as in the example from [6]. \end{remark}\begin{theorem}The formula for the \emph{3-SAT} problem with four clauses in the variables $x_i$ and $y_j$, negated $(x_i. y_j.)$ or not $x_i y_j$, same $x_i=x_j$ or different written in $\mathcal{PB}\ni g(p',q',0)$ are for any case in this set up \begin{align}((((x_1 x_2 x_3).  (x_4 x_5 x_6).).).(((y_{6} y_{5} y_{4}).(y_{3} y_{2} y_{1}). ).).)'=0\end{align} when $(p\lor q \mapsto pq) \land (p\land q \mapsto p.q.'),$ by three conjunctive compositions, left-left, right-right and left-right, following the disjunctive multiplications. For the general case, with three literals in each clause (3-SAT), in $g(p',q',0)$ we get \begin{align}((\cdots (((x_1 x_2 x_3). (x_4 x_5x_6).).).\cdots ((((x_i x_j x_k). (x_{\ell} x_m x_n).).)\cdots ).).\\ \nonumber ((\cdots ((y_{n} y_{m} y_{\ell}).  (y_{k} y_{j} y_{i}).).\cdots (((y_{6}y_{5} y_{4}).  (y_{3} y_{2} y_{1}).).).). \cdots ).)'=0 \end{align} following the pattern established in the 4-clauses case, and each composition gives  parentheses of the form:  $ "\ (( \  $  $  \ ).)' \ " $.  The left side is without dots and $'$, and they come on the right side, ending the formula. If $ ").)'"$ is put inside, we just use dots. \end{theorem} \begin{proof} This rests on the multiplication $2n$ times of the polynomial in the $n-$clauses of the type $\hat p \hat q,$ and a composition of $p.q.'$ as many times as there are clauses minus 1, $2n-1.$ For the family, we took $g(p',q',0)$ with $p \lor q \mapsto pq$ and $p \land q \mapsto (p+1)(q+1)'=p.q.'$, as we need the simplest possible composition within the clauses.\end{proof} We shall now generalize the SAT-problem to any length in clauses of the type $S^j_{c(\lor,\ell>3)}=(p\lor q \lor r \lor s \lor \cdots \lor s_{\ell})$ and to any number of literals $|p, q, r, s  \cdots s_{\ell}|=\ell \in \mathbb{N}$ of any number $n \in \mathbb{N}$ of conjunctions where exponent $1,2\cdots n$ just shows the position of the clause $S^j_{c(\lor,\ell>3)}$ in the line counting both from $(1,2,\cdots ,n-1, n)$ the left, and from $(n, n-1,\cdots ,2,1)$ the right.   \begin{align} S^1_{c(\lor,\ell>3)} \land S^2_{c(\lor,\ell>3)} \land \ldots \land S^{j-1}_{c(\lor,\ell>3)} \land S^j_{c(\lor,\ell>3)}\land\ldots \land S^{n-1}_{c(\lor,\ell>3)} \land S^n_{c(\lor,\ell>3)}\\ \  \land  S^{n}_{c(\lor,\ell>3)} \land S^{n-1}_{c(\lor,\ell>3)}\land \ldots \land S^j_{c(\lor,\ell>3)} \land S^{j-1}_{c(\lor,\ell>3)} \land \ldots S^2_{c(\lor,\ell>3)}\land S^1_{c(\lor,\ell>3)}. \end{align}  \begin{theorem} After using the polynomial transfer the isomorphism, $g:\mathcal{OP}\to \mathcal{PB}$ for the clauses, by multiplication $$g(S^j_{c(\lor,\ell>3)})=s_g^j=\hat p \hat q=(pqrs...p_\ell)_{\hat p}(pqrs...p_\ell)_{\hat q}$$ where each literal can be placed in any other clause, reversed(negated) or not, alone or with other literals in an arbitrary way, we get a binary mirrored automorphism $g:\mathcal{PB} \times \mathcal{PB} \to \mathcal{PB} $ composition $g(\hat s^1_g \cdots \hat s^n_g \circ \hat s^n_g \cdots \hat s^1_g )$  based on the transfer function for disjunction defined for $$g_{(\land)}(p',q',0)(\hat p \hat q)=((\hat p.)(\hat q.))'=\hat p.\hat q.'$$ composing now polynomial arguments binary  $(s^{j-1}_g \circ s^j_g) \in \mathcal{PB}.$ \begin{align}g^0(s_g^j)=((pqrs\cdots s_{\ell}).0.)'=(pqrs\cdots s_{\ell})= s_g^j;  \\ g^1=(s_g^1 \circ s_g^2)=(pqrs\cdots s_{\ell}) \circ (pqrs\cdots s_{\ell})  \\ =((pqrs\cdots s_{\ell}).(pqrs...s_{\ell}).)'=(s_g^1.s_g^2.)'\\ \nonumber
g^2=(s_g^1.s_g^2.).\circ (s_g^3.s_g^4.)'=(((s_g^1.s_g^2.).).(((s_g^3.s_g^4.).).)'\\ g^{1,2,\cdots j} \cdots \cdots \cdots g^{j\cdots ,2,1} \\ \nonumber
g^{1,2,\cdots n}=((s_g^1. s_g^2.).(s_g^3. s_g^{4}.).\ \cdots (s_g^{n-3}. s_g^{n-2,1}.).) (((s_g^{n-1}. s_g^n.).).)\cdots .)'  \\
 \ g^{n\cdots ,2,1}=(((s_g^n. s_g^{n-1}.).).( s_g^{n-2}. s_g^{n-3}.).).  \ \cdots (( s^4_g.s_g^3.).( s_g^2.  s_g^1.).).)'\\
g^{1,2,\cdots n} \circ g^{n\cdots ,2,1} =\\
(\cdots(((( s_g^1.  s_g^2.). \ \cdots ( s_g^{n-1}.  s_g^n.).).)\cdots .).  ( s_g^n.  s_g^{n-1}.).\cdots .( s_g^2. s_g^1.).).).)\cdots .)'=0  
\end{align} \end{theorem} \begin{theorem} i) If the clauses are \emph{negative repetitive} such that for all $(pqrs\cdots s_{\ell})\cap (pqrs\cdots s_{\ell}) \not = \emptyset$ with different parity then we get the reduction $$(\hat s_g^i \land \hat s_g^j) \mapsto \hat s_g^i+\hat s_g^j \hookrightarrow (\hat S^i \equiv \hat S^j)\in \emph{Op}_{(..)}'$$ $$\hat s_g=(pqrs\cdots s_{\ell})_1+(pqrs\cdots s_{\ell})_2+\cdots +(pqrs\cdots s_{\ell})_n$$ and after renumbering the two mirrored series in $(3.30)$ to one \begin{align}=s_g^1+ s_g^2+ s_g^3+ s_g^{4}+\cdots +s_g^{n-3}+ s_g^{n-2} + s_g^{n-1}+ s_g^n=0\end{align}
ii) The method in the first part (i) can be used although $\sum_i s_g^i \sum_j s_g^j \not = 0$ in an \emph{exclusive solution} belonging to \emph{ker}$_{xor} \hat s_g$ where we disregard solutions coming from the intersection, now a closed area, that has $(0\cdots 1)\in g(p',q',0)$ or $(1\cdots 0)\in g(p,q,1)$ as \emph{ker}$_{xor} \hat s_g \subset$ \emph{ker} $\hat s_g$.\end{theorem} \begin{proof} This follows from the definition of the composition $\hat p.\hat q.'$ when we find different parity as in $(p.qrs\cdots s_{\ell})(pq.rs\cdots s_{\ell})=0$. \end{proof}
\begin{corollary} If \emph{Theorem 3.15} is not satisfied, so $\sum_i s_g^{i} \sum_j s_g^{j} \not = 0$, we can use one sum, say $\sum_{j} s^{j}=0,$
which, in a step-by-step conjugation, gives all the variables in this sum. We can substitute these values into $\sum_i s_g^{i}$ to solve for the unknown variables, again using step-by-step conjugation. \end{corollary} \begin{proof} By (3.10) and case 3, we can write $\sum_i s_g^{i}\sum_j s_g^{j}.+\sum_{j} s^{j}_g=0.$ If we multiply with the conjugate $(\sum_{j} s^{j}_g)$ we get $\sum_i s_g^{i} \sum_j s_g^{j}.(\sum_{j} s^{j}_g)=\sum_{j} s^{j}_g (\sum_{j} s^{j}_g) $ and $ \sum_{j} s^{j}_g=0$, by Boolean contradictory multiplication and idempotence, and we find all literals in $\sum_{j} s^{j}_g$ and in $\sum_{i} s^{i}_g.$
\end{proof} We use this method in our final, more complicated, Example 3.19. Still another possibility is the following \begin{corollary} If \emph{Theorem 3.15} is not satisfied, so $\sum_i s_g^{i} \sum_j s_g^{j} \not = 0$, a fully saturated term $p.q.r.s. \cdots s_{\ell}$ can be ignored, then a trivial solution $(0000\cdots 0)$ must also be ignored.  \end{corollary} \begin{proof} As $p.q.r.s. \cdots s_{\ell}$ must be zero, we can ignore this term to avoid contradiction. If we remove the term, a trivial solution becomes possible, since the contradiction condition in  $p.q.r.s. \cdots s_{\ell}.=1$ is now not longer there.
\end{proof} \begin{example} Prove that this SAT-problem have just one solution\begin{align} (p \lor q \lor r) \land (p \lor q \lor \lnot r) \land (p \lor \lnot q \lor r) \land  ( p \lor \lnot q \lor \lnot r)\land \\ \nonumber (\lnot p \lor q \lor r)\land(\lnot p \lor q \lor \lnot r) \dot \land(\lnot p \lor \lnot q \lor r) \end{align} \begin{proof} This SAT is odd so we put the principal operator (with a dot) in the last conjunction, and then compose the odd clause with two already composed clauses on the next turn. We have as before multiplication for the clauses and $ \land \mapsto_{(g,'p,'0)} p.q.'$ for the compositions $$(p q r \circ p q r.)\circ  (p q. r\circ p q. r.)\circ  (p. q r\circ  p. q r.) \dot \circ  (p. q. r)$$ $$=((p q r). (p q r.).)\circ  ((p q. r). (p q. r.).).\dot\circ  ((p. q r). (p. q r.).).\circ  (p. q. r)$$ Here we stop the composing process, because we can use Theorem 3.15; for the complete composed formula, see (3.37) below. $$=(pqr+pqr.)\circ (pq.r+pq.r.)\circ (p.qr+p.qr.)\circ (p. q. r)$$ $$=(pqr + pqr.+ pq.r + pq.r.)\circ (p.qr+p.qr.+ p.q.r)$$ $$=(pqr + pqr.+ pq.r + pq.r.+p.qr+p.qr.+ p.q.r)$$ Here we found an example of a totally additive result for the final composition. We factorize and reduce $$=pq(r + r.)+ pq.(r + r.)+p.q(r+r.)+ p.q.r$$ $$=pq+ pq.+p.q+ p.q.r=pq+ p+q+ p.q.r=p.q.+ p.q.r=1 \equiv p.q.r.=1$$ and by adding 1 to both sides and factoring out $(p.q.)$, we prove we got only the solution (000) in $g(p',q',0).$ This time the saturated term gave us the solution directly. We note here that the CNF formula is a case of symmetry $$\chi_1  \chi_2  \chi_3 +  \chi_1.  \chi_2  \chi_3 + \chi_1  \chi_2.  \chi_3+\chi_1  \chi_2 \chi_3. $$ $$ + \chi_1  \chi_2. \chi_3.+\chi_1.  \chi_2. \chi_3+\chi_1.  \chi_2. \chi_3= \chi_1.  \chi_2. \chi_3.=1$$ because the CNF is just the expansion of $\chi_1.  \chi_2. \chi_3.$ filling up the I$^2_{\ell=3}$. The CNF here is a case of Theorem 3.15, case i) as it has the complete null space, ker PBE=(000). \end{proof} The final formula without reductions became \begin{align}(p q r \circ p q r.)\circ  (p q. r\circ p q. r.)\circ  (p. q r\circ  p. q r.)  \circ  (p. q. r)\\=((p q r). (p q r.).)\circ  ((p q. r). (p q. r.).).\dot\circ  ((p. q r). (p. q r.).).\circ  (p. q. r)\\ =((((p q r). (p q r.). (((p q. r). (p q. r.).).).).\circ ((((p. q r). (p. q r.).).).(p. q. r).).\\=((((((p q r). (p q r.).). (((p q. r). (p q. r.).).).).).(((((p. q r). (p. q r.).).).(p. q. r).).).)=1.\end{align} \end{example}In the next example, we shall decide if it is SAT or UNSAT \begin{example} Here we look into a problem written in $g(p',q',0)$ and we shall decide the nature of this SAT.\begin{align}(x_1 x_2. x_3. \circ  x_1. x_2. x_3.)\circ (x_2.x_3.\circ x_2x_3)  \circ (x_3x_4 \circ x_3 x_4.) \\ \nonumber \circ (x_2. x_3. \circ x_2 x_3)    \circ (x_3 x_4\circ x_3 x_4.)\end{align}\begin{proof} We start by doing five compositions and reductions in the order above.\begin{align}(x_1 x_2. x_3. \circ  x_1. x_2. x_3.)=x_1 x_2. x_3.+ x_1. x_2. x_3.=  x_2.x_3.(x_1+x_1.)=x_2.x_3.\\
(x_2.x_3.\circ x_2x_3)=x_2.x_3.+ x_2x_3\\
(x_3x_4 \circ x_3 x_4.)= x_3x_4 + x_3 x_4.=x_3 \\
( x_2. x_3. \circ x_2 x_3)=x_2. x_3.+ x_2 x_3\\
(x_3 x_4\circ x_3 x_4.)=x_3 x_4+ x_3 x_4.=x_3
\end{align} This was the first round of composition; now we compose (3.39) and (3.40): $$x_2.x_3.(x_2.x_3.+ x_2x_3)+ x_2.x_3. +x_2.x_3.+ x_2x_3=x_2.x_3.+ x_2x_3$$ and (3.41) and (3.42) give $x_2. x_3.+ x_3$ which we compose  with $x_3$ from (3.43) to get $x_3(x_2. x_3.+ x_3)+x_3+ x_2. x_3.+ x_3=x_2. x_3.+ x_3$ and finally $$(x_2.x_3.+ x_2x_3)(x_2. x_3.+ x_3)+ x_3+ x_2x_3=x_2. x_3.+x_3=0 \equiv x_2x_3 =x_2.$$ which gives $(x_1,x_2, x_3, x_4)=(x_1, 1, 0, x_4)$. \end{proof}\end{example} So this is SAT, and the example is from [7], where it was presented in the operator space \begin{align}(x_1 \lor \lnot x_2 \lor \lnot x_3)\land (\lnot x_1 \lor \lnot x_2 \lor \lnot x_3) \land (x_2\lor x_3)\land(x_3\lor x_4)\\ \nonumber \land(x_3 \lor \lnot x_4)  \land (\lnot x_2 \lor \lnot x_3) \land (x_2\lor x_3)\land (x_3\lor x_4)\land (x_3 \lor \lnot x_4). \end{align} \begin{example} Now we look at an example using the additive model.  We have reversed repetition of literals at the beginning, giving only addition until last composition produces one non original higher order phrase $ab.c.d$  by multiplication underlined  \begin{align}(a.b c + a c d + a c d. +  \underline{a c. d} \ \circ \ a c.d. +  \underline{b. c. d} +a.b c. + a.b. c) \\ = ab.c.d+a.b c + a c d + a c d. +  a c. d + a c.d. +  b. c. d +a.b c. + a.b. c \\=ab.c.d+bc.d+ac.d.\\ =ac.(b.d+d.)+bc.d= ac.(bd).+bc.d\equiv ac.(bd+1)=bdc. \\ \text{ker} \hat s_g= \{ab1d, 00cd, 0bc0\} \\ \text{by reductions below} \end{align}\begin{equation}\left\{\begin{array}{ll}
1.\ a.b c +a.b c.=a.b(c+c.)=a.b \\
2. \ a c d + a c d.=ac; \ 2+3+4=ac+bc.d+ac.d.+a.c \\ =bc.d + ac.d.+c=bc.d+ac.d.\\
3. \ b c. d + a c.d.\\
4.  \  a.b c. + a.b. c=a.(bc.+b.c)=a.b+a.c
\end{array}\right.\end{equation}
We also give the full composition as in (3.31), to compare  $$(((((a.b c). (a c d)).). (((a c.d.). (b. c. d).).).).). $$ $$(((((a c.d.). ((b. c. d).).). (((a.b c.). (a.b. c).).).).).)'=0.$$\end{example} A harder SAT problem could be the following, presented as a composition of phrases in $g(p',q',0)$, and here we use Corollary 3.15.2 \begin{example}\begin{align}x_1  x_2  x_3.\circ x_4.  x_5.  x_6. \circ x_3 x_8\circ x_3  x_8. \circ x_4   x_5.  \circ x_9.  x_{10}. \circ x_1 x_9 \circ x_{7} x_{8}\end{align} We compose the phrases that will give fewer higher terms first, as we may rearrange them to the following \begin{align} ((x_1 x_2 x_3.\circ x_3 x_8) \circ (x_4. x_5. x_6. \circ x_4 x_5.)) \circ ((x_9.  x_{10}. \circ x_1 x_9)\circ (x_{7} x_{8}.\circ x_3  x_8.))=0 \end{align} Here $\ell=10 $ in eight phrases, giving seven compositions. \begin{equation}\left\{ \begin{array}{ll}
1. x_1  x_2   x_3.+x_3 x_8  \\ 2. \ x_4.  x_5.  x_6.+x_4   x_5. \\ 3.\ x_9.  x_{10} +x_1 x_9 \\ 4. \ x_3 x_{7} x_8. +x_3  x_8.+x_{7} x_{8}.= x_3. x_7x_8. +x_3x_8.\\ \end{array}\right. \end{equation} We compose $(1.\circ 2.)$ using  $ (\sum_i s_g^{i})(\sum_j s_g^{j}).+\sum_{j} s^{j}=0$   $$(x_1 x_2 x_3.+x_3 x_8) (x_4. x_5. x_6.+x_4   x_5.). +x_4. x_5. x_6.+x_4   x_5. $$ and then $(3.\circ 4.)$  $$=(x_9.  x_{10} +x_1 x_9)(x_3. x_7x_8. +x_3x_8.).+ x_3. x_7 x_8. +x_3x_8.$$ and finally the last composition $(1. 2.\circ 3. 4.)$ \begin{align}((x_1 x_2 x_3.+x_3 x_8) (x_4. x_5. x_6.+x_4   x_5.). +x_4. x_5. x_6.+x_4   x_5.)\\ ((x_9.  x_{10} +x_1 x_9)(x_3. x_7x_8. +x_3x_8.).+ x_3. x_7 x_8. +x_3x_8.).\\ +(x_9.  x_{10} +x_1 x_9)(x_3. x_7x_8. +x_3x_8.).+ x_3. x_7 x_8. +x_3x_8.=0\end{align} Again, we used (3.10) number 3, and for the equality $(\sum_i s_g^{i})(\sum_j s_g^{j}).=\sum_{j} s^{j}$ we use Corollary 3.15.1 to get $$ \sum_{j} s^{j}_g=(x_9.  x_{10} +x_1 x_9)(x_3. x_7x_8. +x_3x_8.).+ (x_3. x_7 x_8. +x_3x_8.)=0.$$ Again, we multiply by the right-hand-side conjugate and $x_3. x_7 x_8. +x_3x_8.=0$ and take $x_3=0,  x_7=0$ and $x_8=1$, and we use that in (3.55) \begin{align}x_1 x_2  (x_4. x_5. x_6.+x_4 x_5.). +x_4. x_5. x_6.+x_4   x_5. =0,\end{align} and we find again $x_5.(x_4.  x_6.+x_4) =0.$
Now we have $x_3=0, x_4=0, x_5=0, x_6=1,x_7=0,x_8=1$ and finally by testing $x_1=0,x_2=x_2,x_9=1$ and $x_{10}=1$ we get the string in $g(p', q', 0)$,$$(0_1x_20_30_40_51_60_71_81_9 x_{10})$$ which is a YES instance for \begin{align}(X_1 \lor X_2 \lor \bar X_3)\land (\bar X_4 \lor \bar X_5 \lor \bar X_6) \land (X_3 \lor X_8) \land (X_3 \lor \bar X_8)  \land (X_4 \lor \bar X_5) \\ \nonumber \land (\bar X_9 \lor  \bar X_{10}) \land (X_1\lor X_9) \land (X_{7}\lor X_{8}),\end{align}
which is a problem solvable by manual calculation. \end{example} We conclude with the calculation rules, giving reductions (3.60)-(3.69) or augmentations (3.70), we observed so in our study.\begin{lemma} Under composition of clauses $\chi_{j=1,2,3}$ by conjunctions in $g(p',q',0)$ we get \begin{align}
\chi_1 \chi_2 + \chi_1 \chi_2= 0 \\
\chi_1 \chi_2 + \chi_1 \chi_2.= \chi_1 \\
\chi_1 \chi_2 + \chi_1. \chi_2=  \chi_2 \\
\chi_1. \chi_2 + \chi_1 \chi_2.= \chi_1+ \chi_2 \\
\chi_1 \chi_2 + \chi_1. \chi_2.= (\chi_1 +\chi_2).\\
\chi_1 \chi_2. + \chi_1. \chi_2.=\chi_2.\\
\chi_1 \chi_2. \chi_3.+ \chi_1. \chi_2. \chi_3.=\chi_2.\chi_3.\\
\chi_1 \chi_2+\chi_1. \chi_2+\chi_1 \chi_2.+\chi_1. \chi_2.=1 \\
\chi_1\chi_2+\chi_1\chi_3+ \chi_2 \chi_3  =\chi_1.\chi_2.+\chi_1.\chi_3.+\chi_2. \chi_3. +1 =0\\
\chi_1.  \chi_2  \chi_3 +  \chi_1  \chi_2.  \chi_3 + \chi_1  \chi_2  \chi_3.+\chi_1 . \chi_2  . \chi_3. =\chi_1+\chi_2+\chi_3\\
\chi_1=  (\chi_1 \lor \chi_2)\land (\chi_1 \lor \lnot\chi_2) \mapsto (\chi_1 \chi_2).(\chi_1 \chi_2.).' =\chi_1
\end{align} .\end{lemma}  
\begin{thebibliography}{99}
\bibitem{Cu} Bergmann, A.: \emph{An Introduction to Many-valued and
Fuzzy Logic}, Cambridge University Press, 2008.
\bibitem{Cu} Borgeke, P.B.: \emph{PBNF-transform as a formulation of propositional Calculus, I}, ArXiv, 2026.
\bibitem{Cu} Borgeke, P.B.: \emph{PBNF-transform as a formulation of propositional Calculus, II}, ArXiv, 2026.
\bibitem{Cu} Church, A.: \emph{Introduction to Mathematical Logic},
Princeton University Press, 2008.
\bibitem{Co} Cook, S.: \emph{The P versus NP problem}, Clay Mathematics Institute, 2013.
\bibitem{Cu} J. Horacek, M. Kreuzer, \emph{On conversions from CNF to ANF}, Journal of Symbolic Computation 100, 2020, 164-186.
\bibitem{Cu} Kalla, P.: \emph{The Boolean Satisfiability (SAT) Problem, SAT Solver Technology, and Equivalence Verification}, Slides on homepage, August 30, 2021.
\bibitem{Cu} Kelly, L.J.: \emph{General Topology}, Van Nostrand Reinhold Company, 1961.
\bibitem{ps} Quine, W.V,.: \emph{Mathematical Logic}, W.W. Norton
Company Inc., 1940.
\bibitem{Cu}  Sansone, E., Solar-Lezama, A.:\emph{Power Term Polynomial Algebra for Boolean Logic}, 17th Pragmatics of SAT International Workshop, July 19, 2026, Lisbon, Portugal.
\bibitem{Cu}  Sikorski, R.: \emph{Boolean Algebras}, Springer-Verlag, Berlin, Heidelberg, New York, 1969.
\end{thebibliography}
\end{document}